\documentclass{amsart}
\usepackage{amsfonts,amssymb,amscd,amsmath,latexsym,amsbsy,bm}
\numberwithin{equation}{section}
\theoremstyle{plain}

\newtheorem{thm}{Theorem}[section]

\newtheorem{prop}[thm]{Proposition}

\newtheorem{defi}[thm]{Definition}
\newtheorem{lem}[thm]{Lemma}

\theoremstyle{remark}
\newtheorem{rema}[thm]{Remark}

\newcommand{\Z}{\mathbb{Z}}

\newcommand{\C}{\mathbb{C}}

\title[Boundary quantum KZ equations]{Boundary quantum Knizhnik-Zamolodchikov equations and fusion}

\author{Nicolai Reshetikhin}
\address{Department of Mathematics, University of California, Berkeley,
CA 94720, USA, KdV Institute for Mathematics, University of Amsterdam,
Science Park 904, 1098 XH Amsterdam, The Netherlands \& ITMO University, Kronverkskii ave. 49, Saint Petersburg, 197101, Russia.}
\email{reshetik@math.berkeley.edu}

\author{Jasper Stokman}
\address{KdV Institute for Mathematics, University of Amsterdam, Science Park 904, 1098 XH Amsterdam, The Netherlands \& IMAPP, Radboud University, Heyendaalseweg 135, 6525 AJ Nijmegen, The Netherlands.}
\email{j.v.stokman@uva.nl}

\author{Bart Vlaar}
\address{KdV Institute for Mathematics, University of Amsterdam, Science Park 904, 1098 XH Amsterdam, The Netherlands.}
\email{b.h.m.vlaar@uva.nl}

\begin{document}
\begin{abstract}
In this paper we extend our previous results concerning Jackson integral solutions of the boundary quantum Knizhnik-Zamolodchikov equations with diagonal K-operators to higher-spin representations of quantum affine $\mathfrak{sl}_2$. 
First we give a systematic exposition of known results on $R$-operators acting in the tensor product of evaluation representations in Verma modules over quantum $\mathfrak{sl}_2$.
We develop the corresponding fusion of $K$-operators, which we use to construct diagonal $K$-operators in these representations. 
We construct Jackson integral solutions of the associated boundary quantum Knizhnik-Zamolodchikov equations and explain how in the finite-dimensional case they can be obtained from our previous results by the fusion procedure.
\end{abstract}
\maketitle
\setcounter{tocdepth}{1}
\tableofcontents
\section{Introduction}

The boundary q-Knizhnik-Zamolodchikov (qKZ) equations have their origins in the representation theory through works of Cherednik \cite{CQKZ,CQKZ2}
and in quantum field theory and in statistical mechanics with special "integrable" boundary conditions, see, e.g., \cite{BPT,JKKMW,dF,dFZJ,W}. 
For detailed references see \cite{RSV}.
Their formulation involves solutions to the Yang-Baxter equation, the so-called $R$-operators or $R$-matrices, and solutions to the reflection equation, known as (boundary) $K$-operators or $K$-matrices. 

\subsection{The boundary qKZ equations}
Let $M^\ell$ be the Verma module over quantum $\mathfrak{sl}_2$ with highest weight $\ell \in \C$.
Then we will denote by $\mathcal{R}^{k\ell}(x)$ the operator acting in $M^k\otimes M^\ell$ which is the evaluation of the truncated
universal $R$-matrix for quantum affine $\mathfrak{sl}_2$ acting in the tensor product of corresponding evaluation representations. 
It satisfies the Yang-Baxter equation:
\begin{equation}\label{qYBkl}
\mathcal{R}^{k\ell}_{12}(x-y)\mathcal{R}^{km}_{13}(x-z)\mathcal{R}^{\ell m}_{23}(y-z)=
\mathcal{R}^{\ell m}_{23}(y-z)\mathcal{R}^{km}_{13}(x-z)\mathcal{R}^{k\ell}_{12}(x-y)
\end{equation}
This is an equation in $M^k\otimes M^\ell\otimes M^m$ and we are using the standard notations
$R^{k\ell}_{12}(x)=R^{k \ell}(x)\otimes \mathrm{Id}_{M^m}$, etc.
For details and references see Section \ref{ic}.

Given the above $R$-operator $\mathcal{R}^{k\ell}(x)$, operators $\mathcal{K}^{+,\ell}(x)$ and $\mathcal{K}^{-,\ell}(x)$  acting in $M^\ell$ are called left and right $K$-operators if they satisfy the left and right reflection equations, respectively. These equations are also known as ``boundary Yang Baxter equations'' and were introduced in \cite{Sk}. 
In the current setting they are given by
\begin{equation}
\label{rree-gen}
 \begin{aligned}
& \mathcal{R}^{k\ell}(x-y)\mathcal{K}_{1}^{+,k}(x)\mathcal{R}^{\ell k}_{21}(x+y)\mathcal{K}_{2}^{+,\ell}(y) = \\
& \qquad = \mathcal{K}_{2}^{+,\ell}(y)\mathcal{R}^{k\ell}(x+y)\mathcal{K}_{1}^{+,k}(x)\mathcal{R}^{\ell k}_{21}(x-y), \\
& \mathcal{R}^{\ell k}_{21}(x-y)\mathcal{K}_{1}^{-,k}(x)\mathcal{R}^{k\ell}(x+y)\mathcal{K}_{2}^{-,\ell}(y) = \\
& \qquad = \mathcal{K}_{2}^{-,\ell}(y)\mathcal{R}^{\ell k}_{21}(x+y)\mathcal{K}_{1}^{-,k}(x)\mathcal{R}^{k\ell}(x-y). 
\end{aligned}
\end{equation}
These are equations in $M^k\otimes M^\ell$; we are using the notations $\mathcal{K}_{1}^{\pm,k}(x)=\mathcal{K}^{\pm,k}(x)\otimes \textup{Id}$, $\mathcal{K}_{2}^{\pm, \ell}(y)=
\textup{Id}\otimes \mathcal{K}^{\pm, \ell}(y)$ and $\mathcal{R}_{21}^{\ell k}(x) := \mathcal{P}^{\ell k}\mathcal{R}^{\ell k}(x) \mathcal{P}^{k\ell}$, where $\mathcal{P}^{k \ell}: M^{k}\otimes M^{\ell}\to M^{\ell}\otimes M^{k}$ is the permutation operator $\mathcal{P}^{k \ell}(m^k \otimes m^\ell)=m^\ell \otimes m^k $ ($m^k  \in M^k$, $m^\ell \in M^\ell$). 

For $\underline{\ell} = (\ell_1,\ldots,\ell_N) \in \C^N$, consider the tensor product
\[ M^{\underline{\ell}}=M^{\ell_1}\otimes \cdots \otimes M^{\ell_N}. \]
The boundary qKZ equations \cite{CQKZ,CQKZ2} in $M^{\underline{\ell}}$ are given by the following compatible system of difference equations
\begin{equation}\label{ReflqKZ}
f(\mathbf{t}+\tau \bm e_r)=
\Xi_r^{\underline{\ell}}(\mathbf{t};\xi_+,\xi_-;\tau)f(\mathbf{t}),\qquad r=1,\ldots,N
\end{equation}
for $M^{\underline{\ell}}$-valued meromorphic functions $f(\mathbf{t})$ in $\mathbf{t}\in\mathbb{C}^N$, where $\tau\in\mathbb{C}^\times$ and $\{\bm e_r\}_r$ is the standard orthonormal
basis of $\mathbb{R}^N$. Here
\begin{equation}\label{Atauj}
\begin{split}
\Xi_r^{\underline{\ell}}(\mathbf{t};\xi_+,\xi_-;\tau)&:=
\mathcal{R}_{r,r+1}^{\ell_r\ell_{r+1}}(t_r-t_{r+1}+\tau)\cdots \mathcal{R}_{r,N}^{\ell_r\ell_N}(t_r-t_N+\tau)\\
&\qquad \times \mathcal{K}_{r}^{+,\ell_r}(t_r+\frac{\tau}{2})
\mathcal{R}_{N,r}^{\ell_N\ell_r}(t_N+t_r)\cdots \mathcal{R}_{r+1,r}^{\ell_{r+1}\ell_r}(t_{r+1}+t_r)\\
&\qquad \times \mathcal{R}_{r-1,r}^{\ell_{r-1}\ell_r}(t_{r-1}+t_r)\cdots
\mathcal{R}_{1,r}^{\ell_1\ell_r}(t_1+t_r)\mathcal{K}_{r}^{-,\ell_r}(t_r)\\
&\qquad \times \mathcal{R}_{r,1}^{\ell_r\ell_1}(t_r-t_1)\cdots \mathcal{R}_{r,r-1}^{\ell_r\ell_{r-1}}(t_r-t_{r-1})
\end{split}
\end{equation}
is the (boundary) transport operator on $M^{\underline{\ell}}$, depending meromorphically on $\mathbf{t}\in\mathbb{C}^N$. 
The compatibility of the system \eqref{ReflqKZ} is guaranteed by the conditions
\[ \Xi_r^{\underline{\ell}}(\mathbf{t}+\mathbf{e}_s \tau;\xi_+,\xi_-;\tau)\Xi_s^{\underline{\ell}}(\mathbf{t};\xi_+,\xi_-;\tau) = \Xi_s^{\underline{\ell}}(\mathbf{t}+\mathbf{e}_r \tau;\xi_+,\xi_-;\tau)\Xi_r^{\underline{\ell}}(\mathbf{t};\xi_+,\xi_-;\tau), \]
for $r,s=1,\ldots,N$, which themselves are consequences of the quantum Yang-Baxter and reflection equations (\ref{qYBkl}-\ref{rree-gen}). In this paper we construct explicit Jackson integral solutions of \eqref{ReflqKZ} when the left and right $K$-operators 
$\mathcal{K}^{\pm,\ell}(x)$ are of the form $\mathcal{K}^{\xi_{\pm},\ell}(x)$ with $\mathcal{K}^{\xi,\ell}(x)$ ($\xi\in\mathbb{C}$) an explicit
one-parameter family of $K$-operators diagonal with respect to the weight basis of $M^\ell$.

\subsection{Finite-dimensional representations and fusion}

When $\ell \in \frac{1}{2} \Z_{\geq 0}$ the representation $M^\ell$ is no longer irreducible; it has an infinite-dimensional subrepresentation and an irreducible finite-dimensional quotient representation $V^\ell$. 
When some of the $\ell_s$'s in the boundary qKZ equations are in $\frac{1}{2}\Z_{\geq 0}$ the equations descend to the tensor product of corresponding quotient modules.

For $k,\ell \in \frac{1}{2} \Z_{\geq 0}$, the tensor product of the associated evaluation modules $V^k(x) \otimes V^\ell(y)$ becomes reducible for special values of $x, y \in \C$
\cite{CP}. Owing to this degeneracy, $R$-operators acting in (tensor products of) higher-dimensional evaluation modules can be obtained from corresponding objects acting in (tensor products of) lower-dimensional evaluation modules
through a process called fusion \cite{KRS,EFK}. 
We extend this re\-pre\-sen\-ta\-tion-theoretic approach to fusion of $K$-operators in Section \ref{RE4section}. 
Such $R$- and $K$-operators can then be generalized to $R$- and $K$-operators associated with 
modules $M^\ell(x)$ for arbitrary $\ell \in \C$ by means of an analytical continuation.
This will allow us to establish the above reflection equation \eqref{rree-gen} for a larger class of $K$-operators than hitherto has been done. In particular, we obtain the diagonal
$K$-operators $\mathcal{K}^{\xi,\ell}(x)$ from this fusion approach applied to Cherednik's \cite{CQKZ2} diagonal $K$-matrix associated to $V^{\frac{1}{2}}$. The 
$\mathcal{K}^{\xi,\ell}(x)$ are closely
related to the family of $K$-operators constructed in \cite{DN} using the $q$-Onsager algebra.

For other approaches to fusion of $K$-operators, see e.g. \cite{DKM,KS,IOZ,MN,MNR,Z}.

\subsection{Main result}

In \cite{RSV} we constructed $q$-integral solutions to \eqref{ReflqKZ} when all $\ell_s=\frac{1}{2}$.
In this case the corresponding irreducible quotient spaces are two-dimensional and \eqref{ReflqKZ} reduces to an equation in $(\C^2)^{\otimes N}$.
The main result of this paper is the construction of $q$-integral solutions to (\ref{ReflqKZ}) for arbitrary $\ell_s \in \C$. For $\ell_s \in \frac{1}{2}\Z_{\geq 0}$ it gives Jackson integral
solutions in the tensor product of corresponding irreducible representations $V^{\ell_s}$. 
Our main result (Theorem \ref{mr}) can be summarized as follows.
\begin{thm}\label{mrintro}
Let $\xi_+,\xi_-\in\mathbb{C}$ and let $g_{\xi_+,\xi_-}(x)$, $h(x)$
and $F^{\ell}(x)$ be meromorphic functions in $x\in\mathbb{C}$ satisfying
the functional equations
\begin{equation*}
\begin{split}
g_{\xi_+,\xi_-}(x+\tau)&= \frac{\sinh(\xi_--x-\frac{\eta}{2})\sinh(\xi_+-x-\frac{\tau}{2}-\frac{\eta}{2})}{\sinh(\xi_-+x+\tau-\frac{\eta}{2}) \sinh(\xi_++x+\frac{\tau}{2}-\frac{\eta}{2})}g_{\xi_+,\xi_-}(x),\\
h(x+\tau)&=\frac{\sinh(x+\tau)\sinh(x+\eta)}{\sinh(x)\sinh(x+\tau-\eta)}h(x),\\
F^\ell(x+\tau)&=\frac{\sinh(x+\tau-\ell\eta)}{\sinh(x+\tau+\ell\eta)}F^\ell(x).
\end{split}
\end{equation*}
Given fixed generic $\mathbf{x}_0\in\mathbb{C}^S$, and fixed parameters $\xi_+,\xi_-,\eta,\tau$ in a suitable parameter domain (see Section \ref{sec:mr}),
the 
$M^{\underline{\ell}}$-valued sum
\begin{equation*}
\begin{split}
f_S^{\underline{\ell}}(\mathbf{t}):=\sum_{\mathbf{x}\in\mathbf{x}_0+\tau\mathbb{Z}^S}
&\Bigl(\prod_{i=1}^Sg_{\xi_+,\xi_-}(x_i)\Bigr)
\Bigl(\prod_{1\leq i<j\leq S}h(x_i+ x_j)h(x_i-x_j)\Bigr)\\
&\times\Bigl(\prod_{r=1}^N\prod_{i=1}^S F^{\ell_r}(t_r+ x_i) F^{\ell_r}(t_r- x_i)\Bigr)
\Bigl( \prod_{i=1}^S\overline{\mathcal{B}}^{\xi_-}(x_i;\mathbf{t}) \Bigr) \Omega
\end{split}
\end{equation*}
is a solution of the boundary qKZ equations \eqref{ReflqKZ}, meromorphic in $\mathbf{t}\in\mathbb{C}^N$.
Here, $\overline{\mathcal{B}}^\xi(x;\mathbf{t})$ are matrix elements of the boundary quantum monodromy matrix 
and $\Omega=m_1^{\ell_1} \otimes \cdots \otimes m_1^{\ell_N}$ is the tensor product of highest-weight vectors $m_1^{\ell_s} \in M^{\ell_s}$ (see Section \ref{Bethevectors} for details). 
\end{thm}
Explicit formulae for functions $g_{\xi_+,\xi_-}, h$ and $F^\ell$ are given in Section \ref{sec:mr}. 
We will discuss integral (not Jackson integral) solutions in a forthcoming paper. It yields a complete system of solutions to the boundary qKZ equations.

Theorem \ref{mrintro} gives for $\ell_s\in\frac{1}{2}\mathbb{Z}_{\geq 0}$ Jackson integral solutions of the boundary qKZ equations taking values in
\[
V^{\underline{\ell}}=V^{\ell_1}\otimes\cdots\otimes V^{\ell_S}.
\]
These can alternatively be obtained from a fusion procedure applied to the Jackson integral
solutions when all $\ell_s=\frac{1}{2}$ derived earlier in \cite{RSV} (see Subsection \ref{fusionsubsection}). 
It seems though that the result for continuous spin $\ell_s\in\mathbb{C}$
(Theorem \ref{mrintro}) cannot be obtained from half-integer spins by analytic continuation.

\subsection{Outline of the paper}
In Section \ref{ic} and Section \ref{FusionSection-R}
we overview solutions to the quantum Yang-Baxter equation corresponding to quantum $\mathfrak{sl}_2$ and their fusion,
following \cite{KR,KRS, EFK}. Reflection equations and the fusion of $K$-operators are discussed in Section \ref{HSK-matrices}.
The boundary monodromy matrices defined in terms of these $R$- and $K$-operators are introduced in Section \ref{Bethevectors}, as are the off-shell Bethe vectors $\bigl( \prod_{i=1}^S\overline{\mathcal{B}}^{\xi_-}(x_i;\mathbf{t}) \bigr) \Omega$.
In Section \ref{sec:mr} we state and discuss the main theorem on the Jackson integral solutions of the boundary qKZ equations with continuous spins; its proof is given in Section \ref{sec:proof}. 
In Section \ref{sec:fusion2}, we show that the boundary qKZ equations \eqref{ReflqKZ} acting on $V^{\underline \ell}$ ($\ell_s\in\frac{1}{2}\mathbb{Z}_{\geq 0}$)
and the associated Jackson integral solutions of the 
boundary qKZ equations can be obtained from the special case when all $\ell_s=\frac{1}{2}$ by fusion.

\subsection{Acknowledgements} 
The research of N.R. was supported by Chern-Simons chair and by the NSF grant DMS-1201391; he also acknowledges support from the KdV Institute of the University of Amsterdam and from QGM at Aarhus, where an important part of the work has been done. 
J.S. and B.V. are grateful to the University of California for hospitality; the work of B.V. was supported by an NWO free competition grant.

\section{Quantum affine $\mathfrak{sl}_2$ and $R$-operators}\label{ic}
In this section we discuss basic facts on quantum affine $\mathfrak{sl}_2$ and its associated evaluation $R$-and $L$-operators, following \cite{FR,EFK}.
We use slightly different conventions compared to \cite{FR,EFK} in order to obtain a direct match with the $R$-and $L$-operators of the
$6$-vertex model (see Subsection \ref{Loperator}).

\subsection{Quantum affine algebra $\mathfrak{sl}_2$ and the universal $R$-matrix}
We fix $\eta\in\mathbb{C}$ such that $p:=\mathrm{e}^\eta$ is not a root of unity. We write $p^x:=\mathrm{e}^{\eta x}$ for $x\in\mathbb{C}$.

Set $\mathfrak{h}=\mathbb{C}h_0\oplus\mathbb{C}h_1$. Quantum affine $\mathfrak{sl}_2$ is the Hopf algebra  $\widehat{\mathcal{U}}_\eta:=
\mathcal{U}_\eta(\widehat{\mathfrak{sl}}_2)$ over $\mathbb{C}$
with generators $e_i, f_i$ ($i=0,1$), $p^{h}$ ($h\in\mathfrak{h}$) and with defining relations
\begin{gather*}
p^0=1,\qquad p^{h+h^\prime}=p^hp^{h^\prime},\\
p^{h}e_ip^{-h}=p^{\alpha_i(h)}e_i,\quad p^hf_ip^{-h}=p^{-\alpha_i(h)}f_i,
\quad \lbrack e_i,f_j\rbrack=\delta_{i,j}\frac{p^{h_i}-p^{-h_i}}{p-p^{-1}},\\
\begin{aligned} e_i^3e_j-(p^2+1+p^{-2})e_i^2e_je_i+(p^2+1+p^{-2})e_ie_je_i^2-e_je_i^3&=0,&& i\not=j,\\
f_i^3f_j-(p^2+1+p^{-2})f_i^2f_jf_i+(p^2+1+p^{-2})f_if_jf_i^2-f_jf_i^3&=0,&& i\not=j
\end{aligned}
\end{gather*}
for $i,j=0,1$ and $h,h^\prime\in\mathfrak{h}$.
Here $\alpha_i$ are linear functionals on $\mathfrak{h}$ satisfying $\alpha_j(h_i)=a_{ij}$ with Cartan matrix
\[
\left(\begin{matrix} a_{00} & a_{01}\\ a_{10} & a_{11}\end{matrix}\right)=\left(\begin{matrix} 2 & -2\\ -2 & 2\end{matrix}\right).
\]
The comultiplication $\Delta$ and the counit $\epsilon$ are determined by their
action on generators:
\begin{equation*}
\begin{split}
\Delta(p^h)&=p^h\otimes p^h,\\
\Delta(e_i)&=e_i\otimes 1+p^{-h_i}\otimes e_i,\\
\Delta(f_i)&=f_i\otimes p^{h_i}+1\otimes f_i
\end{split}
\end{equation*}
and
\begin{equation*}
\epsilon(p^h)=1,\quad \epsilon(e_i)=0,\quad \epsilon(f_i)=0.
\end{equation*}
The antipode is determined by $S(p^h)=p^{-h}$, $S(e_i)=-p^{h_i}e_i$ and $S(f_i)=-f_ip^{-h_i}$.

The extension $\widetilde{\mathcal{U}}_\eta$ of this algebra by generators $p^{\lambda d}$ ($\lambda\in\mathbb{C}$)
such that $[p^{\lambda d},p^{h}]=0$ and $p^{\lambda d}e_i=p^{\lambda\delta_{i,0}}e_ip^{\lambda d}$, $p^{\lambda d}f_i=
p^{-\lambda\delta_{i,0}}f_ip^{\lambda d}$
is a quantized Kac-Moody algebra. 
The corresponding Lie algebra has a non-degenerate scalar product and there is a universal $R$-matrix $R\in \widetilde{\mathcal{U}}_\eta\hat{\otimes} \widetilde{\mathcal{U}}_\eta$ \cite{Dr}. 
It has the form
\[
R=\exp(\eta (c\otimes d+d\otimes c))\mathcal{R}
\]
where $c=h_0+h_1$ and $\mathcal{R}\in\widehat{\mathcal{U}}_\eta\hat{\otimes}\widehat{\mathcal{U}}_\eta$.
In the category of modules where $c$ acts by zero (zero-level representations), the
element $\mathcal{R}$ satisfies all properties of the universal $R$-matrix:
\begin{gather*}
\mathcal{R}\Delta(a)=\Delta^\mathrm{op}(a)\mathcal{R}, \\
(\Delta\otimes \textup{Id})(\mathcal{R})=\mathcal{R}_{13}\mathcal{R}_{23}, \qquad (\textup{Id}\otimes \Delta)(\mathcal{R})=\mathcal{R}_{13}\mathcal{R}_{12}. 
\end{gather*}
Here, $\Delta^{op}$ the opposite comultiplication. See also \cite[Lecture 9]{EFK} for further details (note though that we have a different convention for
the comultiplication).

\subsection{Evaluation representations}

We write $\mathcal{U}_\eta\subset \widehat{\mathcal{U}}_\eta$ for the Hopf subalgebra generated by $e_1,f_1$ and $p^{\lambda h_1}$ ($\lambda\in\mathbb{C}$). 
It is the quantized universal enveloping algebra of $\mathfrak{sl}_2$.

Let $\ell\in\mathbb{C}$ and $M^\ell:=\bigoplus_{n=1}^{\infty}\mathbb{C}m_n^\ell$ be a left $\mathcal{U}_\eta$-module
with the action given by
\[ \begin{split}
\pi^\ell(p^{\lambda h_1})m_n^\ell&=p^{2\lambda (\ell+1-n)}m_n^\ell,\\
\pi^\ell(e_1)m_n^\ell&=\frac{\sinh((n-1)\eta)\sinh((2\ell+2-n)\eta)}{\sinh(\eta)^2}m_{n-1}^\ell,\\
\pi^\ell(f_1)m_n^\ell&=m_{n+1}^\ell,
\end{split} \]
where $m_0^\ell:=0$.
The $\mathcal{U}_\eta$-module $(\pi^\ell,M^\ell)$ is the Verma module
with highest weight $\ell$ and highest weight vector $m_1^\ell$.

If $k\in\frac{1}{2}\mathbb{Z}_{\geq 0}$ the subspace $N^k:=\bigoplus_{n=2k+2}^\infty\mathbb{C}m_n^k\subset M^k$ is a $\mathcal{U}_\eta$-submodule. 
We write $V^k:=M^k/N^k$ for the resulting quotient $\mathcal{U}_\eta$-module. The cosets $v_n^k:=m_n^k+N^k$ ($1\leq n\leq 2k+1$) form a weight basis in $V^k$.
The associated representation map will be denoted by $\overline{\pi}^k$ and for this representation of
$\mathcal{U}_\eta$ we will write $(\overline{\pi}^k, V^k)$.

For each $x\in\mathbb{C}$ there exists a unique unit-preserving algebra homomorphism $\phi_x: \widehat{\mathcal{U}}_\eta\rightarrow\mathcal{U}_\eta$ satisfying
\begin{align*}
\phi_x(p^{\lambda h_0})&=p^{-\lambda h_1},& \phi_x(p^{\lambda h_1})&=p^{\lambda h_1},\\
\phi_x(e_0)&=\mathrm{e}^{-x}f_1,& \phi_x(e_1)&=\mathrm{e}^{-x}e_1,\\
\phi_x(f_0)&=\mathrm{e}^xe_1,& \phi_x(f_1)&=\mathrm{e}^xf_1.
\end{align*}
Given a representation $\pi$ of $\mathcal{U}_\eta$ on $V$
we write $\pi_x:=\pi\circ\phi_x$, which turns $V$ in a 
representation of $\widehat{\mathcal{U}}_\eta$ called the evaluation representation.
Sometimes we will denote it by $V(x)$.

In what follows we will work with evaluation representation $(\overline{\pi}^k_x, V^k)$ and $(\pi^\ell_x, M^\ell)$,
where $k \in \frac{1}{2} \Z_{\geq 0}$ and $\ell\in \C$.

\subsection{Evaluation $R$-and $L$-operators} 

We follow here \cite[Lecture 9]{EFK}.
Fix $x,y\in\mathbb{C}$ with $\Re(x-y)\ll 0$. For $k,\ell\in\mathbb{C}$ the evaluation of the truncated universal $R$-matrix
\[
\bigl(\pi^k_x\otimes\pi^\ell_y\bigr)(\mathcal{R})
\]
is a linear operator on $M^k\otimes M^\ell$ which only depends on the difference $x-y$ of $x$ and $y$.
It acts on the tensor product of highest weight vectors as
\[
\bigl(\pi^k_x\otimes\pi^\ell_y)(\mathcal{R})m_1^k\otimes m_1^\ell=\alpha^{k\ell}(x-y)m_1^k\otimes m_1^\ell
\]
where $\alpha^{k\ell}(x-y)$ is invertible for generic $p$ and $x-y$. 
Define
\[
\mathcal{R}^{k\ell}(x-y):=\alpha^{k\ell}(x-y)^{-1}\bigl(\pi^k_x\otimes\pi^\ell_y\bigr)(\mathcal{R}).
\]

The operator $\mathcal{R}^{k\ell}(x-y)$ intertwines the action of $\widehat{\mathcal{U}}_\eta$ with its opposite:
\begin{equation}\label{intertwiningproperty}
\mathcal{R}^{k\ell}(x-y)(\pi^k_x\otimes\pi^\ell_y)(\Delta(X))=(\pi^k_x\otimes\pi^\ell_y)(\Delta^\mathrm{op}(X))\mathcal{R}^{k\ell}(x-y),\qquad
X\in\widehat{\mathcal{U}}_{\eta}
\end{equation}
and satisfies $\mathcal{R}^{k\ell}(x-y)m_1^k\otimes m_1^\ell=m_1^k\otimes m_1^\ell$. These properties determine $\mathcal{R}^{k\ell}(x-y)$
uniquely for generic values of $x-y$. 

The dependence of the operator $\mathcal{R}^{k\ell}(x-y)$ on $x,y,k,\ell$ is as a rational function in $\mathrm{e}^{x-y}, p^k$ and $p^\ell$.
Analytic continuation thus gives a well-defined linear operator $\mathcal{R}^{k\ell}(x-y)$ on $M^k\otimes M^\ell$ for generic values of $x-y$, which can be characterized by the same intertwining property \eqref{intertwiningproperty} with respect to the action of $\widehat{\mathcal{U}}_\eta$.

Let $k\in\frac{1}{2}\mathbb{Z}_{\geq 0}$ and write $\textup{pr}^k: M_k\twoheadrightarrow V_k$ for the canonical map. For each
$x\in\mathbb{C}$, it defines an intertwiner $\textup{pr}^k_x: M_k(x)\twoheadrightarrow V_k(x)$ of $\widehat{\mathcal{U}}_{\eta}$-modules.
Note that for $k\in\frac{1}{2}\mathbb{Z}_{\geq 0}$, there exists a unique linear map 
\[
L^{k\ell}(x-y): V^k\otimes M^\ell\rightarrow V^k\otimes M^\ell\]
depending rationally on $\mathrm{e}^{x-y}$ and satisfying
\[
(\textup{pr}^k\otimes\textup{Id}_{M^\ell})\mathcal{R}^{k\ell}(x-y)=L^{k\ell}(x-y)(\textup{pr}^k\otimes\textup{Id}_{M^\ell}).
\]
Similarly, for $k,\ell\in\frac{1}{2}\mathbb{Z}_{\geq 0}$, there exists a unique linear map 
\[
R^{k\ell}(x-y): V^k\otimes V^\ell\rightarrow
V^k\otimes V^\ell
\]
satisfying
\begin{equation}\label{finitered}
(\textup{pr}^k\otimes\textup{pr}^{\ell})\mathcal{R}^{k\ell}(x-y)=R^{k\ell}(x-y)(\textup{pr}^k\otimes\textup{pr}^{\ell}).
\end{equation}

\subsection{Basic properties of evaluation $R$-and $L$-operators}
We follow \cite{FR} and for details \cite[Lecture 9]{EFK}.

The basic properties of the universal $R$-matrix give the quantum Yang-Baxter equation
\begin{equation}
\mathcal{R}^{k\ell}_{12}(x-y)\mathcal{R}^{km}_{13}(x-z)\mathcal{R}^{\ell m}_{23}(y-z)=
\mathcal{R}^{\ell m}_{23}(y-z)\mathcal{R}^{km}_{13}(x-z)\mathcal{R}^{k\ell}_{12}(x-y)
\end{equation}
as linear operators on $M^k\otimes M^\ell\otimes M^m$. 
In addition, the operator $\mathcal{R}^{k\ell}(x-y)$ satisfies unitarity:
\[
\mathcal{R}^{k \ell}(x-y)^{-1}= \mathcal{R}^{\ell k}_{21}(y-x),
\]
where 
\[
\mathcal{R}^{\ell k}_{21}(x):=\mathcal{P}^{\ell k}\mathcal{R}^{\ell k}(x)\mathcal{P}^{k\ell}:
M^k\otimes M^{\ell}\rightarrow M^k\otimes M^{\ell}
\]
and $\mathcal{P}^{k\ell}: M^k\otimes M^\ell\rightarrow M^\ell\otimes M^k$ is the permutation operator.

Both properties descend naturally to the $L$-operators and finite $R$-operators. In particular, the familiar RLL-relations
\begin{equation}\label{RLL}
R^{k\ell}_{12}(x-y)L_{13}^{km}(x-z)L_{23}^{\ell m}(y-z)=
L_{23}^{\ell m}(y-z)L_{13}^{km}(x-z)R^{k\ell}_{12}(x-y)
\end{equation}
for $k,\ell\in\frac{1}{2}\mathbb{Z}_{\geq 0}$
as well as the quantum Yang-Baxter equation for the $R$-operators $R^{k\ell}(x)$ ($k,\ell\in\frac{1}{2}\mathbb{Z}_{\geq 0}$) follow immediately
from the quantum Yang-Baxter equation for $\mathcal{R}^{k\ell}$.

The next property of $\mathcal{R}^{k\ell}(x)$ is $P$-symmetry:
\begin{lem}\label{PsymmLem}
As linear maps on $M^k \otimes M^\ell$ we have for generic $x\in\mathbb{C}$,
\begin{equation}\label{Psymmetrytoprove}
\mathcal{R}^{\ell k}_{21}(x)=\mathcal{R}^{kl}(x).
\end{equation}
\end{lem}
\begin{proof}
Write $T^{k\ell}(x)$ for the left-hand side of \eqref{Psymmetrytoprove}. Then clearly
\[
T^{k\ell}(x)m_1^k\otimes m_1^\ell=m_1^k\otimes m_1^\ell=\mathcal{R}^{k\ell}(x)m_1^k\otimes m_1^\ell.
\]
Hence it suffices to show that for generic $x$ and $y$,
\[
T^{k\ell}(x-y)(\pi_x^k\otimes\pi_y^\ell)(\Delta(X))=(\pi_x^k\otimes\pi_y^\ell)(\Delta^{op}(X))T^{k\ell}(x-y),\qquad
\forall\, X\in\widehat{\mathcal{U}}_{\mathrm{e}^\eta}.
\]
This is clear for $X=p^h$ ($h\in\mathfrak{h}$). For $X=e_0,f_1$ it is a direct consequence of the identity
\[
(\pi_y^k\otimes\pi_x^\ell)(\Delta^{op}(e_0))=(\pi_{-y}^k\otimes\pi_{-x}^\ell)(\Delta(f_1))
\]
and \eqref{intertwiningproperty}. For the algebraic generators $X=e_1,f_0$ it follows similarly from
\eqref{intertwiningproperty} using the fact that
\[
(\pi_y^k\otimes\pi_x^\ell)(\Delta^{op}(e_1))=(\pi_{-y}^k\otimes\pi_{-x}^\ell)(\Delta(f_0)). \qedhere
\]
\end{proof}
Finally we discuss crossing symmetry. 
We start with crossing symmetry for $L$-operators:
\begin{lem}\label{cslem}
Let $k\in\frac{1}{2}\mathbb{Z}_{\geq 0}$ and $\ell\in\mathbb{C}$.
Let $w^k: V^k\overset{\sim}{\longrightarrow} V^k$ be the linear isomorphism defined by
\[
w^k(v_n^k):=c_nv_{2k+2-n}^k
\]
with $c_n\in\mathbb{C}^\times$ determined by the recursion $c_{n+1}:=-c_np^{2k+1-2n}$ and $c_1:=1$. 
Then 
\[
L^{k\ell}(-x)^{T_1}=\alpha^{k\ell}(x)\alpha^{k\ell}(x-\eta)(w^k\otimes\textup{Id}_{M^\ell})L^{k\ell}(x-\eta)(w^k\otimes\textup{Id}_{M^\ell})^{-1}
\]
with $T_1$ the transpose in the first tensor component with respect to the weight basis.
\end{lem}
\begin{proof}
For an evaluation module $(\pi,V)$ over $\widehat{\mathcal{U}}_\eta$ we write $(\pi^*,V^*)$ for the graded dual $V^*$ of $V$ with respect to the weight grading, with $\widehat{\mathcal{U}}_\eta$-action $(\pi^*(X)\phi)(v):=\phi(\pi(S(X))v)$.
If $A: V\rightarrow V$ is a linear map, then we write $A^t: V^*\rightarrow V^*$ for the corresponding dual linear operator.

It follows from the identity $(S\otimes \textup{Id})(R)=R^{-1}$ that
\begin{equation}\label{crossingunitarity1}
\bigl((\overline{\pi}^k_x)^*\otimes\pi^\ell_y\bigr)(\mathcal{R})=\bigl((\overline{\pi}^k_x\otimes\pi^\ell_y)(\mathcal{R}^{-1})\bigr)^{t_1}.
\end{equation}
Here $t_1$ means taking the dual with respect to the first component in the tensor product. 
Write $\{(v_n^k)^*\}$ for the basis of $(V^k)^*$ dual to the weight basis $\{v_n^k\}_n$ of $V^k$. 
We identify $V^k\simeq (V^k)^*$ by $v_n^k\mapsto (v_n^k)^*$ (the dual $A^t$ of a linear operator $A: V^k\rightarrow V^k$ then corresponds to the
transpose $A^T$ of $A$ with respect to the weight basis $\{v_n^k\}$ of $V^k$). Accordingly we interpret the map
$w^k$ as a linear map $w^k: V^k\rightarrow (V^k)^*$, in which case it defines an isomorphism $V^k(x-\eta)\overset{\sim}{\longrightarrow} V^k(x)^*$ of 
$\widehat{\mathcal{U}}_{\eta}$-modules.
Consequently
\begin{equation*}
\begin{split}
L^{k\ell}(-x+y)^{T_1}&=\alpha^{k\ell}(x-y)\bigl(\overline{\pi}_x^k\otimes\pi_y^\ell\bigr)(\mathcal{R}^{-1})\bigr)^{T_1}\\
&=\alpha^{k\ell}(x-y)\bigl((\overline{\pi}_x^k)^*\otimes\pi^\ell_y\bigr)(\mathcal{R})\\
&=\alpha^{k\ell}(x-y)(w^k\otimes\textup{Id}_{M^\ell})(\overline{\pi}_{x-\eta}^k\otimes\pi_y^\ell)(\mathcal{R})
(w^k\otimes\textup{Id}_{M^\ell})^{-1},
\end{split}
\end{equation*}
where we have used \eqref{crossingunitarity1} for the second equality. This proves the desired result.
\end{proof}
\begin{rema}
For $k\in\mathbb{C}$ the canonical linear isomorphism $M^k\overset{\sim}{\longrightarrow} (M^k)^{**}$ defines an isomorphism
$M^k(x-2\eta)\overset{\sim}{\longrightarrow} M^k(x)^{**}$ of $\widehat{\mathcal{U}}_\eta$-modules (cf. Lemma \ref{cslem}). 
It then follows from a double application
of \eqref{crossingunitarity1} (for arbitrary evaluation modules) that
\[
\mathcal{R}^{k\ell}(x-2\eta)=\frac{\alpha^{k\ell}(x)}{\alpha^{k\ell}(x-2\eta)}\bigl(\bigl((\mathcal{R}^{k\ell}(x)^{-1})^{T_1}\bigr)^{-1}\bigr)^{T_1}.
\]
Note the difference with \cite[Prop. 9.5.2]{EFK}, which involves an additional conjugation by a diagonal operator in the first tensor component.
\end{rema}

\subsection{Explicit formulae for $L$-operators} \label{Loperator}
It is possible to compute $L^{\frac{1}{2}\ell}(x)$ explicitly using the expression of the universal $R$-matrix (a comprehensive survey of this can be found in \cite{BGKNR}). 
This leads to the formulae
\begin{equation*}
\begin{split}
L^{\frac{1}{2}\ell}(x)(v_1^{\frac{1}{2}}\otimes m_n^\ell)&=\frac{\sinh(x+(\frac{3}{2}+\ell-n)\eta)}{\sinh(x+(\frac{1}{2}+\ell)\eta)}v_1^{\frac{1}{2}}\otimes m_n^\ell\\
&+\mathrm{e}^{(\ell+\frac{3}{2}-n)\eta}\frac{\sinh((n-1)\eta)\sinh((2\ell+2-n)\eta)}{\sinh(\eta)\sinh(x+(\frac{1}{2}+\ell)\eta)}v_2^{\frac{1}{2}}\otimes m_{n-1}^\ell
\end{split}
\end{equation*}
and
\begin{equation*}
\begin{split}
L^{\frac{1}{2}\ell}(x)(v_2^{\frac{1}{2}}\otimes m_n^\ell)&=
\mathrm{e}^{(-\ell-\frac{1}{2}+n)\eta}\frac{\sinh(\eta)}{\sinh(x+(\ell+\frac{1}{2})\eta)}v_1^{\frac{1}{2}}\otimes m_{n+1}^\ell\\
&+\frac{\sinh(x+(-\frac{1}{2}-\ell+n)\eta)}{\sinh(x+(\frac{1}{2}+\ell)\eta)}v_2^{\frac{1}{2}}\otimes m_n^\ell.
\end{split}
\end{equation*}
Note that exponential factors can be removed by a similarity transformation.
After this, the result 
coincides with the $L$-operator found in \cite{KR}. It follows from these formulae that the finite
$R$-operator $R^{\frac{1}{2}\frac{1}{2}}(x)$ is the $6$-vertex $R$-operator:
\begin{equation}\label{6vertexR}
R^{\frac{1}{2}\frac{1}{2}}(x)
=\frac{1}{\sinh(x+\eta)}\left(\begin{matrix}
\sinh(x+\eta) & 0 & 0 & 0\\
0 & \sinh(x) & \sinh(\eta) & 0\\
0 & \sinh(\eta) & \sinh(x) & 0\\
0 & 0 & 0 & \sinh(x+\eta)\end{matrix}
\right)
\end{equation}
with respect to the ordered basis $(v_1^{\frac{1}{2}}\otimes v_1^{\frac{1}{2}},
v_1^{\frac{1}{2}}\otimes v_2^{\frac{1}{2}}, v_2^{\frac{1}{2}}\otimes v_1^{\frac{1}{2}},
v_2^{\frac{1}{2}}\otimes v_2^{\frac{1}{2}})$ of $V^{\frac{1}{2}}\otimes V^{\frac{1}{2}}$.

The crossing symmetry of the 
$L$-operators (Lemma \ref{cslem}) becomes
\begin{equation}\label{crossingsymmetry}
L^{\frac{1}{2}\ell}(-x)^{T_1}=\vartheta^\ell(x)\sigma^y_1 L^{\frac{1}{2}\ell}(x-\eta)\sigma^y_1
\end{equation}
as linear operators on $V^{\frac{1}{2}}\otimes M^{\ell}$, where $T_1$ is the matrix transpose with respect to the weight basis in $V^{\frac{1}{2}}$ and
\begin{equation}\label{varthetal}
\sigma^y:=\left(\begin{matrix} 0 & -\sqrt{-1}\\ \sqrt{-1} & 0\end{matrix}\right), \ \ \vartheta^\ell(x)=\frac{\sinh(x-(\frac{1}{2}-\ell)\eta)}{\sinh(x-(\frac{1}{2}+\ell)\eta)}.
\end{equation}
Formula \eqref{crossingsymmetry} can be directly verified using the above explicit formulae for $L^{\frac{1}{2} \ell}(x)$.

\section{Fusion of $R$-operators}\label{FusionSection-R}
We use the notations from Section \ref{ic}. Fix a generic $\eta\in\mathbb{C}$ throughout this section and write
$p=\mathrm{e}^\eta$.
\subsection{Tensor products of evaluation representations}
Let $k,\ell\in\frac{1}{2}\mathbb{Z}_{\geq 0}$. By \cite[Thm. 4.8]{CP} the tensor product
$\widehat{\mathcal{U}}_{\eta}$-module $V^k(x)\otimes V^\ell(y)$ is irreducible for generic $x,y\in\mathbb{C}$.
For the fusion of $R$- and $K$-operators we need to focus on the special cases that the
$\widehat{\mathcal{U}}_{\eta}$-module $V^k(x)\otimes V^\ell(y)$ is
reducible.

For $k,\ell\in\frac{1}{2}\mathbb{Z}_{\geq 0}$ we write $P^{k\ell}: V^k\otimes V^\ell\rightarrow
V^\ell\otimes V^k$ for the permutation operator.
The following result should be compared with \cite[Prop. 4.9]{CP}. The proof is by a straightforward computation.
\begin{prop}\label{intertwiners}
Let $k\in\frac{1}{2}\mathbb{Z}_{\geq 0}$. \\
{\bf (i)} The linear map $\iota^k: V^{k+\frac{1}{2}}\hookrightarrow V^{\frac{1}{2}}\otimes V^{k}$, defined by
\[
\iota^k\bigl(v_n^{k+\frac{1}{2}}\bigr)=\mathrm{e}^{\frac{\eta}{2}(n-1)}v_1^{\frac{1}{2}}\otimes v_n^k+\mathrm{e}^{-\frac{\eta}{2}(n-2-2k)}\frac{\sinh((n-1)\eta)}{\sinh(\eta)}
v_2^{\frac{1}{2}}\otimes v_{n-1}^k,
\]
defines a $\widehat{\mathcal{U}}_{\eta}$-intertwiner
$\iota_x^k: V^{k+\frac{1}{2}}(x)\hookrightarrow V^{\frac{1}{2}}(x-k\eta)\otimes V^{k}(x+\frac{\eta}{2})$.\\
{\bf(ii)} The linear map
$j^k:=P^{\frac{1}{2}k}\iota^k: V^{k+\frac{1}{2}}\hookrightarrow V^{k}\otimes V^{\frac{1}{2}}$
defines
a $\widehat{\mathcal{U}}_{\eta}$-intertwiner
\[ j^k_x: V^{k+\frac{1}{2}}(x)\hookrightarrow V^{k}(x-\frac{\eta}{2})\otimes V^{\frac{1}{2}}(x+k\eta).
\]
\end{prop}
Note that the intertwiners $\iota_x^k$ and $j_x^k$ do not depend on $x$ as linear maps. We add the subscript $x$ to
clarify the $\widehat{\mathcal{U}}_{\eta}$-action we are considering.

\subsection{Fusion operators}

It follows from Lemma \ref{PsymmLem} that the $R$-operators $R^{k\ell}(x)$ ($k,\ell\in\frac{1}{2}\mathbb{Z}_{\geq 0}$) are $P$-symmetric. 
In the remainder of this section we focus on the fusion of the $R$-operators $R^{k\ell}(x)$ ($k,\ell\in\frac{1}{2}\mathbb{Z}_{\geq 0}$).

For the fusion of the $R$-operators the interpretation of $R$-operators as intertwiners
between tensor products of evaluation modules plays a crucial role. We need explicit expressions
for its action in case that the tensor product of the evaluation modules is reducible.
\begin{lem}\label{link}
For $k\in\frac{1}{2}\mathbb{Z}_{\geq 0}$ the linear operators $R^{\frac{1}{2}k}(x)$ and
$R^{k\frac{1}{2}}(x)$ are regular at $x=(k+\frac{1}{2})\eta$. The resulting linear maps
$S^k:=P^{\frac{1}{2}k}R^{\frac{1}{2}k}\bigl((k+\frac{1}{2})\eta\bigr)$ and
$T^k:=P^{k\frac{1}{2}}R^{k\frac{1}{2}}\bigl((k+\frac{1}{2})\eta\bigr)$, which we will view as
$\widehat{\mathcal{U}}_{\eta}$-intertwiners
\begin{equation*}
\begin{split}
S_x^k: V^{\frac{1}{2}}(\mathrm{e}^{x+k\eta})\otimes V^k(\mathrm{e}^{x-\frac{\eta}{2}})&\rightarrow V^k(\mathrm{e}^{x-\frac{\eta}{2}})\otimes V^{\frac{1}{2}}(\mathrm{e}^{x+k\eta}),\\
T_x^k: V^k(\mathrm{e}^{x+\frac{\eta}{2}})\otimes V^{\frac{1}{2}}(\mathrm{e}^{x-k\eta})&\rightarrow
V^{\frac{1}{2}}(\mathrm{e}^{x-k\eta})\otimes V^k(\mathrm{e}^{x+\frac{\eta}{2}})
\end{split}
\end{equation*}
are explicitly given by
\begin{equation*}
\begin{split}
S^k(v_1^{\frac{1}{2}}\otimes v_n^k)&=\frac{\sinh((2k+2-n)\eta)}{\sinh((2k+1)\eta)}
\mathrm{e}^{-\frac{\eta}{2}(n-1)}j^k(v_n^{k+\frac{1}{2}}),\\
S^k(v_2^{\frac{1}{2}}\otimes v_n^k)&=\frac{\sinh(\eta)}{\sinh((2k+1)\eta)}
\mathrm{e}^{\frac{\eta}{2}(n-2k-1)}j^k(v_{n+1}^{k+\frac{1}{2}}).
\end{split}
\end{equation*}
\begin{equation*}
\begin{split}
T^k(v_n^k\otimes v_1^{\frac{1}{2}})&=\frac{\sinh((2k+2-n)\eta)}{\sinh((2k+1)\eta)}
\mathrm{e}^{-\frac{\eta}{2}(n-1)}\iota^k(v_n^{k+\frac{1}{2}}),\\
T^k(v_n^k\otimes v_2^{\frac{1}{2}})&=\frac{\sinh(\eta)}{\sinh((2k+1)\eta)}
\mathrm{e}^{\frac{\eta}{2}(n-2k-1)}\iota^k(v_{n+1}^{k+\frac{1}{2}}).
\end{split}
\end{equation*}
\end{lem}
\begin{proof}
By $P$-symmetry we have $R^{k\frac{1}{2}}(x)=P^{\frac{1}{2}k}R^{\frac{1}{2}k}(x)P^{k\frac{1}{2}}$,
and Proposition \ref{intertwiners} gives
$\iota^k=P^{k\frac{1}{2}}j^k$. So it suffices to prove the statement for $S^k$. Using the fact that
$(\textup{Id}_{V^{\frac{1}{2}}}\otimes \textup{pr}^k)L^{\frac{1}{2}k}(x)=R^{\frac{1}{2}k}(x)(\textup{Id}_{V^{\frac{1}{2}}}\otimes\textup{pr}^k)$, Remark \ref{Loperator} gives explicit formulae for $S^k$.
Comparing those formulae with the explicit formulae for $j_x^k$ (see Proposition \ref{intertwiners}) now leads to the desired result.
\end{proof}

\subsection{The fusion formula for the $R$- and $L$-operators}
The fusion formulae for the $R$-operators $R^{k\ell}(x)$ ($k,\ell\in\frac{1}{2}\mathbb{Z}_{\geq 0}$) and $L$-operators
$L^{k\ell}(x)$ ($k\in\frac{1}{2}\mathbb{Z}_{\geq 0}$, $\ell\in\mathbb{C}$)
follow directly from the re\-pre\-sen\-ta\-tion-theoretic considerations of
the previous subsection. Recall the linear map $\iota^k: V^{k+\frac{1}{2}}\hookrightarrow V^{\frac{1}{2}}\otimes V^k$ from
Proposition \ref{intertwiners}.

\begin{prop}\label{FusionL}
For $k\in\frac{1}{2}\mathbb{Z}_{\geq 0}$ and $\ell\in\mathbb{C}$ we have the fusion formula
\[
(\iota^k\otimes \textup{Id}_{M^\ell})L^{k+\frac{1}{2},\ell}(x-y)=
L_{13}^{\frac{1}{2}\ell}(x-k\eta-y)L_{23}^{k\ell}(x+\frac{\eta}{2}-y)(\iota^k\otimes\textup{Id}_{M^\ell})
\]
as linear maps $V^{k+\frac{1}{2}}\otimes M^\ell\rightarrow V^{\frac{1}{2}}\otimes V^k\otimes M^\ell$.
\end{prop}
\begin{proof}
Using the fact that
\[
\bigl(\overline{\pi}_x^{\frac{1}{2}}\otimes\overline{\pi}_y^k\otimes\pi_z^\ell\bigr)(\mathcal{R}_{13}\mathcal{R}_{23})=
\bigl(\overline{\pi}_x^{\frac{1}{2}}\otimes\overline{\pi}_y^k\otimes\pi_z^\ell\bigr)((\Delta\otimes\textup{Id})(\mathcal{R}))
\]
and the intertwining property of $\iota_x^k$ (see Proposition \ref{intertwiners}), gives
\[
L_{13}^{\frac{1}{2}\ell}(x-k\eta-y)L_{23}^{k\ell}(x+\frac{\eta}{2}-y)(\iota_x^k\otimes\textup{Id}_{M^\ell})
=
(\iota_x^k\otimes\textup{Id}_{M^\ell})L^{k+\frac{1}{2},\ell}(x-y)
\]
as linear maps $V^{k+\frac{1}{2}}(x)\otimes M^\ell(y)\rightarrow
V^{\frac{1}{2}}(x-k\eta)\otimes V^k(x+\frac{\eta}{2})\otimes M^\ell(y)$.
The result follows now immediately.
\end{proof}
\begin{rema}
Proposition \ref{FusionL} leads for $k,\ell\in\frac{1}{2}\mathbb{Z}_{\geq 0}$ to the fusion formula
\[
(\iota^k\otimes\textup{Id}_{V^\ell})R^{k+\frac{1}{2},\ell}(x-y)=
R_{13}^{\frac{1}{2}\ell}(x-k\eta-y)R_{23}^{k\ell}(x+\frac{\eta}{2}-y)(\iota^k\otimes\textup{Id}_{V^\ell})
\]
for the $R$-operators.
\end{rema}
\begin{rema}
Another approach to fusion formulae for $L$-operators (originating from \cite{KRS}) is by
specialization of the RLL relations \eqref{RLL} at values of $x-y$ for which
$R_{12}^{k\ell}(x-y)$ is not invertible. For instance, in the present setting \eqref{RLL} gives
\begin{align*}
& (T^k\otimes \textup{Id}_{M^\ell})L_{13}^{k\ell}(x+\frac{\eta}{2}-y)L_{23}^{\frac{1}{2}\ell}(x-k\eta-y)=\\
& \qquad = L_{13}^{\frac{1}{2}\ell}(x-k\eta-y)L_{23}^{k\ell}(x+\frac{\eta}{2}-y)(T^k\otimes\textup{Id}_{M^\ell}),
\end{align*}
which shows directly that the operator $L_{13}^{\frac{1}{2}\ell}(x-k\eta-y)L_{23}^{k\ell}(x+\frac{\eta}{2}-y)$ restricts
to a linear endomorphism on the image of $T^k\otimes\textup{Id}_{M^\ell}$. The resulting linear operator is equivalent
to the fused $L$-operator $L^{k+\frac{1}{2},\ell}(x-y)$ in view of Lemma \ref{link}.
\end{rema}

\section{The reflection equation, fusion of $K$-operators and diagonal $K$-operators}\label{HSK-matrices}\label{RE4section}

\subsection{Reflection equations}\label{reintro}
A collection of linear maps $\mathcal{K}^\ell(x): M^\ell\to M^\ell$ is called a family of higher-spin $K$-operators 
if they satisfy the reflection equations in $M^k\otimes M^\ell$:
\begin{equation} \label{reMM}
\mathcal{R}^{k\ell}(x-y)\mathcal{K}_1^k(x)\mathcal{R}^{k\ell}(x+y)\mathcal{K}_2^\ell(y)=
\mathcal{K}_2^\ell(y)\mathcal{R}^{k\ell}(x+y)\mathcal{K}_1^k(x)\mathcal{R}^{k\ell}(x-y).
\end{equation}
\begin{rema}\label{onerefleqn} The natural re\-pre\-sen\-ta\-tion-theoretic forms of the reflection equations \eqref{reMM} involve 
$\mathcal{R}_{21}^{\ell k}(x)=\mathcal{P}^{\ell k}\mathcal{R}^{\ell k}(x)\mathcal{P}^{k\ell}$, cf. \eqref{rree-gen}. 
However, the $P$-symmetry \eqref{Psymmetrytoprove} 
of the $R$-operators has the simplifying effect that all $R$-operators 
can be put into the form $\mathcal{R}^{k\ell}$ and consequently the distinction between left and right 
versions of reflection equations disappears (cf. \cite{Sk}).
\end{rema}
Suppose that for $k\in\frac{1}{2}\mathbb{Z}_{\geq 0}$ there exists a (necessarily unique) linear map
$K^k(x): V^k\rightarrow V^k$ such that 
\[
\textup{pr}^k\circ\mathcal{K}^k(x)=K^k(x)\circ\textup{pr}^k.
\]
Then the equations \eqref{reMM} naturally give rise to (semi-)finite-dimensional versions 
which will also be referred to as reflection equations. 
More precisely, when $k\in \frac{1}{2}\Z_{\geq 0}$ equation \eqref{reMM} projects to the following equation in $V^k \otimes M^\ell$:
\begin{equation} \label{reVM}
L^{k\ell}(x-y) K_1^k(x) L^{k\ell}(x+y)\mathcal{K}_2^\ell(y)=\mathcal{K}_2^\ell(y) L^{k\ell}(x+y) K_1^k(x) L^{k\ell}(x-y);
\end{equation}
Furthermore, when $k,l\in \frac{1}{2}\Z_{\geq 0}$ equation \eqref{reMM} then
projects to the following equation in $V^k \otimes V^\ell$:
\begin{equation} \label{reVV}
R^{k\ell}(x-y) K_1^k(x) R^{k\ell}(x+y) K_2^\ell(y) = K_2^\ell(y) R^{k\ell}(x+y) K_1^k(x) R^{k\ell}(x-y).
\end{equation}

Just as solutions to the quantum Yang-Baxter equation are related to the representation theory of quantized universal enveloping algebras, solutions to the reflection equation ($K$-operators) are related to co-ideal subalgebras of quantized universal enveloping algebras. 
We will discuss it briefly in Subsection \ref{RErep}.

\subsection{$K$-matrices for spin-$\frac{1}{2}$}

With respect to 
the $6$-vertex $R$-operator $R^{\frac{1}{2}\frac{1}{2}}(x)$ (see \eqref{6vertexR}), 
the general diagonal solution of \eqref{reVV} (for $k = \ell = \frac{1}{2}$) is given by Cherednik's \cite{CQKZ2} one-parameter family
\[ K^{\xi,\frac{1}{2}}(x)=\left(\begin{matrix} 1 & 0\\ 0 & \frac{\sinh(\xi-x)}{\sinh(\xi+x)}\end{matrix}\right) \]
written with respect to the basis $(v_1^{\frac{1}{2}},v_2^{\frac{1}{2}})$ of $V^{\frac{1}{2}}$.
To simplify notations we will use $R(x)$ for $R^{\frac{1}{2}\frac{1}{2}}(x)$ and $K^\xi(x)$ for
$K^{\xi,\frac{1}{2}}(x)$. In other words, this matrix acts on the weight basis as
\[
K^\xi(x)v_1^{\frac{1}{2}}=v_1^{\frac{1}{2}},\qquad
K^\xi(x)v_2^{\frac{1}{2}}=\frac{\sinh(\xi-x)}{\sinh(\xi+x)}v_2^{\frac{1}{2}}.
\]
\begin{rema} The proof that $K^\xi(x)$ satisfies \eqref{reVV} for $k=\ell=1/2$ reduces to the identity
\[ \sum_{\epsilon_1,\epsilon_2\in\{\pm 1\}}\epsilon_1\epsilon_2 \frac{\sinh(\xi+\epsilon_1x)\sinh(\xi+\epsilon_2y)} {\sinh(\epsilon_1x +\epsilon_2y)}=0
\]
cf. \cite{RSV}.
\end{rema}
The reflection operator $K^\xi(x)$ satisfies the boundary crossing symmetry:
\begin{equation}\label{Kcrossing}
\textup{Tr}_{2}\Bigl(R_{12}(2x-2\eta)P_{12}K_{2}^\xi(x)\Bigr)=
\frac{\sinh(\xi+x-\eta)\sinh(2x)}{\sinh(\xi+x)\sinh(2x-\eta)}K_1^\xi(x-\eta),
\end{equation}
where $\textup{Tr}_{2}$ is the partial trace over the second tensor component of $V^{\frac{1}{2}}\otimes V^{\frac{1}{2}}$ and $P = P^{\frac{1}{2} \frac{1}{2}}$.
The identity \eqref{Kcrossing} is equivalent to the trigonometric identity
\begin{equation}\label{coeffcond4}
\sinh(\xi+x)\sinh(x-z)+\sinh(\xi-x)\sinh(x+z)=
\sinh(\xi-z)\sinh(2x).
\end{equation}
In Lemma \ref{multivariable} we prove a multivariate extension of \eqref{coeffcond4}, which plays an important
role in the proof of the main result (Theorem \ref{mr}).

A three-parameter family of solutions $K^{\frac{1}{2}}(x)$ of \eqref{reVV} (with $k=\ell=\frac{1}{2}$) is known, see
\cite{dVG,Ne}. 

\subsection{Fusion formula for $K$-operators when $k,\ell\in \frac{1}{2}\Z_{\geq 0}$}\label{FusionKsection}
Notwithstanding Remark \ref{onerefleqn}, in order to put formulas in the natural re\-pre\-sen\-ta\-tion-theoretic form,
we will sometimes use the notation $R_{21}^{\ell k}(x)$.
The intertwining property of the $R$-operator $R^{k\ell}(x)$ gives
\[
R_{21}^{k\ell}(x-y)(\pi_{-x}^\ell\otimes\pi_{-y}^k)(\Delta^{op}(X))=(\pi_{-x}^\ell\otimes\pi_{-y}^k)(\Delta(X))R_{21}^{k\ell}(x-y),
\quad \forall\, X\in\widehat{\mathcal{U}}_{\eta}.
\]

\begin{prop}\label{fusionkprop}
Suppose that the $K^{\frac{1}{2}}(x)$ are complex-linear operators on $V^{\frac{1}{2}}$ depending meromorphically on $x\in\mathbb{C}$
and satisfying the reflection equation
\begin{equation}\label{reqstart}
R^{\frac{1}{2}\frac{1}{2}}_{21}(x-y)K_1^{\frac{1}{2}}(x)R^{\frac{1}{2}\frac{1}{2}}(x+y)K_2^{\frac{1}{2}}(y)=
K_2^{\frac{1}{2}}(y)R^{\frac{1}{2}\frac{1}{2}}_{21}(x+y)K_1^{\frac{1}{2}}(x)R^{\frac{1}{2}\frac{1}{2}}(x-y)
\end{equation}
as linear operators on $V^{\frac{1}{2}}\otimes V^{\frac{1}{2}}$.
Then there exist unique complex-linear operators $K^k(x)$ on $V^k$ for $k\in\frac{1}{2}\mathbb{Z}_{\geq 2}$ satisfying
\begin{equation}\label{fusionk}
j^kK^{k+\frac{1}{2}}(x)=P^{\frac{1}{2}k}K^{\frac{1}{2}}_1(x-k\eta)R^{\frac{1}{2}k}(2x-(k-\frac{1}{2})\eta)
K_2^k(x+\frac{\eta}{2})\iota^k
\end{equation}
for all $k\in\frac{1}{2}\mathbb{Z}_{\geq 2}$.
Furthermore,
\begin{equation}\label{reqfinite}
R^{\ell k}_{21}(x-y)K_1^k(x)R^{k\ell}(x+y)K_2^\ell(y)=
K_2^\ell(y)R_{21}^{\ell k}(x+y)K_1^k(x)R^{k\ell}(x-y)
\end{equation}
as linear operators on $V^k\otimes V^\ell$ for all $k,\ell\in\frac{1}{2}\mathbb{Z}_{>0}$.
\end{prop}

\begin{rema}
We will always set $K^0(x):=\textup{Id}_{V^0}$. 
Then formulae \eqref{fusionk} and \eqref{reqfinite} are trivially satisfied for $k=0$ and/or $\ell=0$.
\end{rema}

\begin{rema}
Fusion of $K$-operators has been studied before in various different contexts, see, e.g., \cite{DKM,KS,IOZ,MNR,MN,MN2,Z}.
\end{rema}
\begin{proof}[Proof of Proposition \ref{fusionkprop}]
Let $m\in\frac{1}{2}\mathbb{Z}_{\geq 0}$ and suppose that the $K$-operators $K^k(x)$ have been constructed for $k\leq m$ satisfying \eqref{fusionk} for
$k<m$ and satisfying \eqref{reqfinite} for $k,l\leq m$.

Consider \eqref{reqfinite} for $\ell=\frac{1}{2}$ and $k=m$, and replace $x$ by $x+\frac{\eta}{2}$ and $y$ by $x-m\eta$. Then we
obtain
\begin{equation*}
\begin{split}
S^mK_2^m(x+\frac{\eta}{2})&\check{R}^{m\frac{1}{2}}(2x-(m-\frac{1}{2})\eta)K_2^{\frac{1}{2}}(x-m\eta)=\\
=&P^{\frac{1}{2}m}K_1^{\frac{1}{2}}(x-m\eta)R^{\frac{1}{2}m}(2x-(m-\frac{1}{2})\eta)K_2^m(x+\frac{\eta}{2})T^m
\end{split}
\end{equation*}
with $\check{R}^{k\ell}(x):=P^{k\ell}R^{k\ell}(x)$
(see Lemma \ref{link} for the definition of $S^m$ and $T^m$). Since the images of the linear maps $T^m$ and $\iota^m$ coincide by Lemma \ref{link}, it follows that
the image of the linear map
\[
P^{\frac{1}{2}m}K_1^{\frac{1}{2}}(x-m\eta)R^{\frac{1}{2}m}(2x-(m-\frac{1}{2})\eta)K_2^m(x+\frac{\eta}{2})\iota^m
\]
is contained in the image of $S^m$. 
By Lemma \ref{link} again, the image of $S^m$ coincides with the image of $j^m$, hence there exists a unique linear operator $K^{m+\frac{1}{2}}(x)$ on $V^{m+\frac{1}{2}}$ such that
\[
j^mK^{m+\frac{1}{2}}(x)=P^{\frac{1}{2}m}K_1^{\frac{1}{2}}(x-m\eta)R^{\frac{1}{2}m}(2x-(m-\frac{1}{2})\eta)
K_2^m(x+\frac{\eta}{2})\iota^m.
\]
It remains to show that \eqref{reqfinite} is valid for $k,\ell\in\frac{1}{2}\mathbb{Z}_{\leq 0}$ and $k,\ell\leq m+\frac{1}{2}$. It suffices to consider the case that
$k=m+\frac{1}{2}$ and/or $\ell=m+\frac{1}{2}$. We divide it into the following three cases:
\renewcommand*\labelenumi{(\theenumi)}
\begin{enumerate}
\item $(k,\ell)=(m+\frac{1}{2},\ell)$ with $\ell\leq m$.
\item $(k,\ell)=(k,m+\frac{1}{2})$ with $k\leq m$.
\item $(k,\ell)=(m+\frac{1}{2},m+\frac{1}{2})$.
\end{enumerate}
If the reflection equation \eqref{reqfinite} is proved for case (1), then (2) follows from (1) using the unitarity of the $R$-operator,
and (3) follows from (1) and (2)
by taking $\ell=m+\frac{1}{2}$ in the following proof of (1).\\

{\it Proof of (1):} Suppose $\ell\in\frac{1}{2}\mathbb{Z}_{\geq 0}$ and $\ell\leq m$. 
Using the fusion formulae of the $R$- and $K$-operators
we obtain
\begin{equation*}
\begin{split}
&R^{\ell,m+\frac{1}{2}}_{21}(x-y)K_1^{m+\frac{1}{2}}(x)R^{m+\frac{1}{2},\ell}(x+y)K_2^\ell(y)=\\
&\qquad=(\iota^m\otimes\textup{Id}_{V^\ell})^{-1}R^{\ell\frac{1}{2}}_{31}(x-m\eta-y)R^{\ell m}_{32}(x+\frac{\eta}{2}-y)
(\iota^m\otimes\textup{Id}_{V^\ell})\\
&\qquad\times
(j^m\otimes\textup{Id}_{V^\ell})^{-1}P^{\frac{1}{2}m}_{12}K_1^{\frac{1}{2}}(x-m\eta)R^{\frac{1}{2}m}_{12}(2x-(m-\frac{1}{2})\eta)
K_2^m(x+\frac{\eta}{2}) 
\\
&\qquad\times 
R^{\frac{1}{2}\ell}_{13}(x-m\eta+y)R^{m\ell}_{23}(x+\frac{\eta}{2}+y)K_3(y)(\iota^m\otimes\textup{Id}_{V^\ell}),
\end{split}
\end{equation*}
where the sublabels $1,2,3$ in the right-hand side stand for the first, second and third tensor component in $V^{\frac{1}{2}}\otimes V^m\otimes V^\ell$
and the sublabels $1,2$ in the left-hand side stand for the first and second tensor component in $V^{m+\frac{1}{2}}\otimes V^\ell$.
Using $P^{m\frac{1}{2}}j^m=\iota^m$ the expression simplifies to
\begin{equation*}
\begin{split}
(\iota^m\otimes\textup{Id}_{V^\ell})^{-1}&R^{\ell\frac{1}{2}}_{31}(x-m\eta-y)K_1^{\frac{1}{2}}(x-m\eta)\\
&\times R_{32}^{\ell m}(x+\frac{\eta}{2}-y)R_{12}^{\frac{1}{2}m}(2x-(m-\frac{1}{2})\eta)R_{13}^{\frac{1}{2}\ell}(x-m\eta+y)\\
&\times K_2^m(x+\frac{\eta}{2})R_{23}^{m\ell}(x+\frac{\eta}{2}+y)K_3^\ell(y)(\iota^m\otimes\textup{Id}_{V^\ell}).
\end{split}
\end{equation*}
Using the quantum Yang-Baxter equation in the second line the expression can be rewritten as
\begin{equation*}
\begin{split}
(\iota^m\otimes\textup{Id}_{V^\ell})^{-1}&R^{\ell\frac{1}{2}}_{31}(x-m\eta-y)K_1^{\frac{1}{2}}(x-m\eta) \\
& \times R_{13}^{\frac{1}{2}\ell}(x-m\eta+y) R_{12}^{\frac{1}{2}m}(2x-(m-\frac{1}{2})\eta)\\
&\times R^{\ell m}_{32}(x+\frac{\eta}{2}-y)K_2(x+\frac{\eta}{2})R_{23}^{m\ell}(x+\frac{\eta}{2}+y)K_3^\ell(y)
(\iota^m\otimes\textup{Id}_{V^\ell}).
\end{split}
\end{equation*}
Applying the reflection equation to the last line leads to the expression
\begin{equation*}
\begin{split}
&(\iota^m\otimes\textup{Id}_{V^\ell})^{-1}R_{31}^{\ell\frac{1}{2}}(x-m\eta-y)K_1^{\frac{1}{2}}(x-m\eta)
R_{13}^{\frac{1}{2}\ell}(x-m\eta+y)K_3^\ell(y)\\
& \qquad \times R_{12}^{\frac{1}{2}m}(2x-(m-\frac{1}{2})\eta)R_{32}^{\ell m}(x+\frac{\eta}{2}+y)\\
& \qquad \times K_2^m(x+\frac{\eta}{2})
R_{23}^{m\ell}(x+\frac{\eta}{2}-y)(\iota^m\otimes\textup{Id}_{V^\ell}).
\end{split}
\end{equation*}
Now applying the reflection equation to the first line gives
\begin{equation*}
\begin{split}
(\iota^m\otimes\textup{Id}_{V^\ell})^{-1}&K_3^\ell(y)R_{31}^{\ell\frac{1}{2}}(x-m\eta+y)K_1^{\frac{1}{2}}(x-m\eta)\\
&\times R_{13}^{\frac{1}{2}\ell}(x-m\eta-y)R_{12}^{\frac{1}{2}m}(2x-(m-\frac{1}{2})\eta)R_{32}^{\ell m}(x+\frac{\eta}{2}+y)\\
&\times K_2^m(x+\frac{\eta}{2})R_{23}^{m\ell}(x+\frac{\eta}{2}-y)(\iota^m\otimes\textup{Id}_{V^\ell}).
\end{split}
\end{equation*}
Applying the quantum Yang-Baxter equation to the second line leads to
\begin{equation*}
\begin{split}
(\iota^m\otimes\textup{Id}_{V^\ell})^{-1}&K_3^\ell(y)R_{31}^{\ell\frac{1}{2}}(x-m\eta+y)R_{32}^{\ell m}(x+\frac{\eta}{2}+y)\\
&\times K_1^{\frac{1}{2}}(x-m\eta)R^{\frac{1}{2}m}_{12}(2x-(m-\frac{1}{2})\eta)K_2^m(x+\frac{\eta}{2})\\
&\times R_{13}^{\frac{1}{2}\ell}(x-m\eta-y)R_{23}^{m\ell}(x+\frac{\eta}{2}-y)(\iota^m\otimes\textup{Id}_{V^\ell}).
\end{split}
\end{equation*}
The fusion formulae for the $R$- and $K$-operators and the fact that $P^{m\frac{1}{2}}j^m=\iota^m$ show that the last expression equals
\[
K_2^\ell(y)R_{21}^{\ell,m+\frac{1}{2}}(x+y)K_1^{m+\frac{1}{2}}(x)R^{m+\frac{1}{2},\ell}_{12}(x-y),
\]
where the sublabels $1$ and $2$ stand for the first and second tensor component in $V^{m+\frac{1}{2}}\otimes V^\ell$. This completes the proof
of the reflection equation for case (1).
\end{proof}

\subsection{Reflection equation and coideal subalgebras}\label{RErep}

Here we briefly discuss the re\-pre\-sen\-ta\-tion-theoretical meaning of reflection equations, cf., e.g., \cite{DMcmp, DM,De}.
Let $\mathcal{A}\subseteq\widehat{\mathcal{U}}_{\eta}$ be a left coideal subalgebra, i.e. it is a unital
subalgebra of $\widehat{\mathcal{U}}_{\eta}$ satisfying $\Delta(\mathcal{A})\subseteq \widehat{\mathcal{U}}_{\eta}\otimes\mathcal{A}$.
If $M$ is a $\widehat{\mathcal{U}}_{\eta}$-module, we write
$M|_{\mathcal{A}}$ for the $\mathcal{A}$-module obtained by restricting the action of $\widehat{\mathcal{U}}_{\eta}$ on $M$ to $\mathcal{A}$.

Suppose that for $k,\ell\in\frac{1}{2}\mathbb{Z}_{>0}$ we have $\mathcal{A}$-intertwiners
\begin{equation}\label{intertwinerA}
K^{k}(x): V^{k}(x)|_{\mathcal{A}}\rightarrow V^{k}(-x)|_{\mathcal{A}},
\qquad K^\ell(x): V^\ell(x)|_{\mathcal{A}}\rightarrow V^\ell(-x)|_{\mathcal{A}}.
\end{equation}
Then the left and right sides of the reflection equation \eqref{reqfinite} are $\mathcal{A}$-intertwiners
$\bigl(V^{k}(x)\otimes V^{\ell}(y)\bigr)|_{\mathcal{A}}\rightarrow \bigl(V^{k}(-x)\otimes V^{\ell}(-y)\bigr)|_{\mathcal{A}}$.
Consequently, if $\bigl(V^{k}(x)\otimes V^{\ell}(y)\bigr)|_{\mathcal{A}}$ is an irreducible $\mathcal{A}$-module for generic $x$ and $y$,
then Schur's lemma implies the reflection equation \eqref{reqfinite} up to a constant. Such examples of $K$-operators have been
constructed with $\mathcal{A}$ the $q$-Onsager algebra, cf., e.g., \cite{De,DMcmp,DM,DN}.

The fusion formula \eqref{fusionk} is compatible with this re\-pre\-sen\-ta\-tion-theoretic perspective in the following sense.  
Assume that $K^{\frac{1}{2}}(x): V^{\frac{1}{2}}(x)|_{\mathcal{A}}\rightarrow V^{\frac{1}{2}}(-x)|_{\mathcal{A}}$
and $K^k(x): V^k(x)|_{\mathcal{A}}\rightarrow V^k(-x)|_{\mathcal{A}}$ are $\mathcal{A}$-intertwiners.
Then the right-hand side of \eqref{fusionk}, which can be written as
\[
K^{\frac{1}{2}}_2(x-k\eta)\check{R}^{\frac{1}{2}k}(2x-(k-\frac{1}{2})\eta))K_2^k(x+\frac{\eta}{2})\iota_x^k
\]
with $\check{R}^{k\ell}(x):=P^{k\ell}R^{k\ell}(x)$, is an $\mathcal{A}$-intertwiner
\[
V^{k+\frac{1}{2}}(x)|_{\mathcal{A}}\rightarrow \bigl(V^k(-x-\frac{\eta}{2})\otimes V^{\frac{1}{2}}(-x+k\eta)\bigr)|_{\mathcal{A}}.
\]
It follows that the corresponding fused $K$-operator $K^{k+\frac{1}{2}}(x): V^{k+\frac{1}{2}}\rightarrow V^{k+\frac{1}{2}}$, 
characterized by
\[
j_{-x}^kK^{k+\frac{1}{2}}(x)=K^{\frac{1}{2}}_2(x-k\eta)\check{R}^{\frac{1}{2}k}(2x-(k-\frac{1}{2})\eta))K_2^k(x+\frac{\eta}{2})\iota_x^k,
\]
becomes an intertwiner
\[
K^{k+\frac{1}{2}}(x): V^{k+\frac{1}{2}}(x)|_{\mathcal{A}}\rightarrow V^{k+\frac{1}{2}}(-x)|_{\mathcal{A}}
\]
of $\mathcal{A}$-modules.

\subsection{Diagonal $K$-operators}\label{diagonalKsection}
\begin{prop}\label{stepKident}
The $K$-operator $K^{\xi,\ell}(x): V^\ell\rightarrow V^\ell$ ($\ell\in\frac{1}{2}\mathbb{Z}_{\geq 0}$) obtained
by recursively fusing $K^{\xi}(x)=K^{\xi,\frac{1}{2}}(x)$ using \eqref{fusionk} acts on the weight basis as
\begin{equation}\label{Kdiagonal}
K^{\xi,\ell}(x)v_n^\ell=C_n^\ell(x;\xi)v_n^\ell,\qquad 1\leq n\leq 2\ell+1,
\end{equation}
where  
\begin{equation}\label{Cell}
C_n^\ell(x;\xi):=\prod_{j=1}^{n-1}\frac{\sinh(\xi-x+(\ell+\frac{1}{2}-j)\eta)}{\sinh(\xi+x+(\ell+\frac{1}{2}-j)\eta)}
\end{equation}
for $n\in\mathbb{Z}_{>1}$ and $C_1^\ell(x;\xi)=1$. 
\end{prop}

\begin{rema}
The $K$-operators $K^{\xi,\ell}(x)$ coincide with an appropriate limit of the explicit $\mathcal{A}$-intertwiner $V^\ell(x)|_{\mathcal{A}}\rightarrow
V^\ell(-x)|_{\mathcal{A}}$ for the $q$-Onsager coideal subalgebra $\mathcal{A}\subset \widehat{\mathcal{U}}_{\eta}$ derived in \cite{DN}. This is to be
expected from the re\-pre\-sen\-ta\-tion-theoretic context of the fusion procedure of $K$-operators, cf. Section \ref{RErep}.
\end{rema}
\begin{proof}[Proof of Proposition \ref{stepKident}]
By induction with respect to $\ell$.
By the fusion formula \eqref{fusionk} for $K$-operators it suffices to show that
\begin{equation}\label{sst1}
C_n^{\ell+\frac{1}{2}}(x;\xi)j^\ell(v_n^{\ell+\frac{1}{2}})=P^{\frac{1}{2}\ell}K_{1}^{\xi,\frac{1}{2}}(x-\ell\eta)R^{\frac{1}{2}\ell}(2x-(\ell-\frac{1}{2})\eta)
K_{2}^{\xi,\ell}(x+\frac{\eta}{2})\iota^{\ell}(v_n^{\ell+\frac{1}{2}})
\end{equation}
with $K^{\xi,\ell}(x)$ satisfying \eqref{Kdiagonal}. Both sides can be computed using the the explicit actions of the maps on the standard bases.
It follows that the desired identity \eqref{sst1} is equivalent to
\begin{equation*}
\begin{split}
C_n^{\ell+\frac{1}{2}}(x;\xi) &=\frac{\sinh(2x+(2-n)\eta)}{\sinh(2x+\eta)}C_n^{\ell}(x+\frac{\eta}{2};\xi)+\\
& \qquad + \frac{\sinh((n-1)\eta)}{\sinh(2x+\eta)}C_{n-1}^{\ell}(x+\frac{\eta}{2};\xi),\\
C_n^{\ell+\frac{1}{2}}(x;\xi) &=\frac{\sinh(\xi-x+\ell\eta)}{\sinh(\xi+x-\ell\eta)} \biggl(\frac{\sinh((2\ell+2-n)\eta)}{\sinh(2x+\eta)}C_n^{\ell}(x+\frac{\eta}{2};\xi)+ \\
& \hspace{32mm} + \frac{\sinh(2x+(n-1-2\ell)\eta)}{\sinh(2x+\eta)}C_{n-1}^{\ell}(x+\frac{\eta}{2};\xi)\biggr)
\end{split}
\end{equation*}
for $1\leq n\leq 2\ell+1$. These follow easily from the trigonometric identity \eqref{coeffcond4}.
\end{proof}

\begin{defi}\label{Kmatrixdef}
For $\ell\in \C$ define the linear operator $\mathcal{K}^{\xi,\ell}(x)$ on $M^\ell$ by
\[
\mathcal{K}^{\xi,\ell}(x)m_n^\ell=C_n^\ell(x;\xi)m_n^\ell,\qquad n\geq 1.
\]
Here functions $C_n^\ell(x;\xi)$ are defined in (\ref{Cell}). 
\end{defi}

Note that if $k\in \frac{1}{2} \Z_{\geq 0}$ and $\text{pr}^k: M^k\to V^k$ is the projection from the Verma module to 
the corresponding finite-dimensional irreducible quotient $V^k$, then
\begin{equation}\label{compatibleK}
\textup{pr}^k\circ\mathcal{K}^{\xi,k}(x)=K^{\xi,k}(x)\circ\textup{pr}^k,
\end{equation}
where $K^{\xi,k}(x): V^k\rightarrow V^k$ is the $K$-operator obtained by fusion in the previous subsection.
\begin{prop}\label{Kdiagonalprop}
Let $\xi\in\mathbb{C}$ then the operators $\mathcal{K}^{\xi,k}(x)$ 
satisfy the reflection equation:
\begin{equation}\label{rree}
\mathcal{R}^{k\ell}(x-y)\mathcal{K}_{1}^{\xi,k}(x)\mathcal{R}^{k\ell}(x+y)\mathcal{K}_{2}^{\xi,\ell}(y)=
\mathcal{K}_{2}^{\xi,\ell}(y)\mathcal{R}^{k\ell}(x+y)\mathcal{K}_{1}^{\xi,k}(x)\mathcal{R}^{k\ell}(x-y)
\end{equation}
for all $k,\ell\in\mathbb{C}$.
\end{prop}

\begin{rema}
From the observations in Subsection \ref{reintro} it follows from Proposition \ref{Kdiagonalprop} that 
for $k \in \frac{1}{2} \Z_{\geq 0}$ and $\ell\in\mathbb{C}$, the $K$-operators $K^{\xi,k}(x)$ and 
$\mathcal K^{\xi,\ell}(x)$ satisfy \eqref{reVM}. 
\end{rema}

\begin{proof}[Proof of Proposition \ref{Kdiagonalprop}] 
For $k,\ell\in\frac{1}{2}\mathbb{Z}_{\geq 0}$ denote by
$d_{n,r;s}^{k,\ell}(\mathrm{e}^x)$ the matrix elements of $R^{k\ell}(x)$ in the weight basis:
\begin{equation}\label{coeffRfinite}
R^{k\ell}(x)v_n^k\otimes v_r^\ell=\sum_{s}d_{n,r;s}^{k,\ell}(\mathrm{e}^x)v_{n-s}^k\otimes v_{r+s}^\ell
\end{equation}
for
 $1\leq n\leq 2k+1$, $1\leq r\leq 2\ell+1$ and $s\in\mathbb{Z}$ such that
$1\leq n-s\leq 2k+1$ and $1\leq r+s\leq 2\ell+1$. Similarly, we write for $k,\ell\in\mathbb{C}$
\begin{equation}\label{coeffR}
\mathcal{R}^{k\ell}(x)m_n^k\otimes m_r^\ell=\sum_{s}c_{n,r;s}(\mathrm{e}^{x}; p^{2k},p^{2\ell})m_{n-s}^k\otimes m_{r+s}^\ell,
\qquad n,r\in\mathbb{Z}_{>0}
\end{equation}
with the sum running over the integers $s$ such that $n-s,r+s\geq 1$. The coefficients 
$c_{n,r;s}(\mathrm{e}^{x};p^{2k},p^{2\ell})$ are rational functions in $\mathrm{e}^{x}$, $p^{2k}$ and 
$p^{2\ell}$. 

Let $n,r\in\mathbb{Z}_{>0}$ satisfying $n-s, r+s\in\mathbb{Z}_{>0}$. Then we have
\begin{equation}\label{linkVermaFinite}
c_{n,r;s}(\mathrm{e}^x;\mathrm{e}^{2\eta k},\mathrm{e}^{2\eta\ell})=d_{n,r;s}^{k,\ell}(\mathrm{e}^x)
\end{equation}
for sufficiently large $k,\ell\in\frac{1}{2}\mathbb{Z}_{>0}$ by \eqref{finitered}. 

Note furthermore that the dependence of $C_n^k(x;\xi)$ on $k$ is by a rational dependence on $p^{2k}$. 
To emphasize it, we write $C_n(x;\xi;p^{2k}):=C_n^k(x;\xi)$ for the remainder of the proof. 

The equation \eqref{rree} we want to prove is equivalent to
the following identities: for all $n,r\in\mathbb{Z}_{>0}$ and $t\in\mathbb{Z}$ satisfying $1-r\leq t\leq n-1$, 
\begin{equation*}
\begin{split}
&\sum_{s=1-r}^{n-1}c_{n-s,r+s;t-s}(\mathrm{e}^{x-y};p^{2k},p^{2\ell})C_{n-s}(x;\xi;p^{2k}) \\
& \qquad \qquad\qquad \times c_{n,r;s}(\mathrm{e}^{x+y};p^{2k},p^{2\ell})C_n(y;\xi;p^{2\ell})=\\
&= \sum_{s=1-r}^{n-1}C_{r+t}(y;\xi;p^{2\ell})c_{n-s,r+s;t-s}(\mathrm{e}^{x+y};p^{2k},p^{2\ell}) \\
& \qquad \qquad \qquad\times C_{n-s}(x;\xi;p^{2k}) c_{n,r;s}(\mathrm{e}^{x-y};p^{2k},p^{2\ell}).
\end{split}
\end{equation*}
Since these identities depend rationally on $p^{2k}$ and $p^{2\ell}$, it suffices to prove them for $k,\ell\in\frac{1}{2}\mathbb{Z}_{\geq 0}$
sufficiently large. But then they follow from \eqref{linkVermaFinite} and the "finite" reflection equations
\[
R^{k\ell}(x-y)K_1^{\xi,k}(x)R^{k\ell}(x_y)K_2^{\xi,\ell}(y)=K_2^{\xi,\ell}(y)R^{k\ell}(x+y)K_1^{\xi,k}(x)R^{k\ell}(x-y)
\]
for $k,\ell\in\frac{1}{2}\mathbb{Z}_{\geq 0}$.
\end{proof}

\section{Boundary monodromy operators and Bethe vectors} \label{Bethevectors}

\subsection{Monodromy matrices} 
In order to formulate our (Jackson integral) solutions to the boundary qKZ equations in $M^{\underline{\ell}} = M^{\ell_1} \otimes \cdots \otimes M^{\ell_N}$ we need to introduce (off-shell) Bethe vectors for the reflecting chain, which in turn are defined using boundary monodromy operators. 
Boundary monodromy operators are linear operators acting on the extended tensor product $V^{\frac{1}{2}} \otimes M^{\underline{\ell}}$; the component $V^{\frac{1}{2}}$ is called auxiliary space and the component $M^{\underline{\ell}}$ state space. {}From now on we restrict
our attention to the case that the $K$-matrices are diagonal (cf. Subsection \ref{diagonalKsection}).

The definition of the boundary monodromy operators involves the $L$-operators
\[
L^\ell(x):=L^{\frac{1}{2}\ell}(x): V^{\frac{1}{2}}\otimes M^\ell\rightarrow V^{\frac{1}{2}}\otimes M^\ell
\]
for $\ell\in\mathbb{C}$.
They provide the link between the integrable structure on the auxiliary space and the integrable structure on the state space and satisfy the RLL commutation relations \eqref{RLL} (with $k=\ell=\frac{1}{2}$ and $R^{\frac{1}{2}\frac{1}{2}}(x)$ the $6$-vertex $R$-operator) as well as the ``mixed'' 
reflection equations \eqref{reVM} (with $k=\frac{1}{2}$, $K^{\frac{1}{2}}(x)=K^\xi(x)$ and 
$\mathcal{K}^{\ell}(x)=\mathcal{K}^{\xi,\ell}(x)$).
In addition,
\begin{equation}\label{compatibleYBE}
L^{k}(x)L^{\ell}(x+y)\mathcal{R}^{k\ell}(y)= \mathcal{R}^{k\ell}(x)L^{\ell}(x+y)L^k(x)
\end{equation}
as linear operators on $V^{\frac{1}{2}}\otimes M^k\otimes M^\ell$.
The $L$-operators $L^{\ell}(x)$, together with the integrable data $K^\xi(x)$ and $R(x)$ on the auxiliary space, define an integrable quantum spin chain with diagonal
reflecting ends (see \cite{Sk}).
It is the inhomogeneous Heisenberg XXZ spin chain with continuous spins.

Let $S_N$ be the symmetric group in $N$ letters.
For $\sigma\in S_N$ define the linear operator $T_\sigma(x;\mathbf{t})=T_\sigma^{\underline{\ell}}(x;\mathbf{t})$
on $V^{\frac{1}{2}}\otimes M^{\underline{\ell}}$ by
\begin{equation} \label{Tdefn}
\begin{split}
T_\sigma(x;\mathbf{t}):=&L^{\ell_{\sigma(1)}}(x-t_{\sigma(1)})
\cdots L^{\ell_{\sigma(N)}}(x-t_{\sigma(N)})\\
=&
\left(\begin{matrix} A_\sigma(x;\mathbf{t}) & B_\sigma(x;\mathbf{t})\\
C_\sigma(x;\mathbf{t}) & D_\sigma(x;\mathbf{t})
\end{matrix}\right),
\end{split}
\end{equation}
where in the last equality we have written $T_\sigma(x;\mathbf{t})$ as a
$\textup{End}(M^{\underline{\ell}})$-valued matrix with respect to the ordered basis
$(v_1^{\frac{1}{2}},v_2^{\frac{1}{2}})$ of $V^{\frac{1}{2}}$.
The special case $T(x;\mathbf{t}):=T_e(x;\mathbf{t})$ with $e\in S_N$ the neutral element is the (A-type) monodomy operator.
We write the corresponding matrix coefficients as $A(x;\mathbf{t})=A_e(x;\mathbf{t}),\ldots,
D(x;\mathbf{t})=D_e(x;\mathbf{t})$.

The operators $T_\sigma(x;\mathbf{t})$ satisfy the commutation relations
\begin{equation}\label{RTT}
R_{00'}(x-y)
T_{\sigma,0}(x;\mathbf{t})T_{\sigma,0^\prime}(y;\mathbf{t})=T_{\sigma,0^\prime}(y;\mathbf{t})T_{\sigma,0}(x;\mathbf{t})R_{00'}(x-y)
\end{equation}
as linear operators on $V^{\frac{1}{2}}\otimes V^{\frac{1}{2}}\otimes M^{\underline{\ell}}$, where $T_{\sigma,0}(x;\mathbf{t})$ is the operator $T_\sigma(x;\mathbf{t})$ acting on the first and third tensor leg and $T_{\sigma,0^\prime}(y;\mathbf{t})$ the operator $T_\sigma(y;\mathbf{t})$ on the second and third tensor leg, while $R_{00'}(x-y)$ is the action of $R(x-y)$
on the tensor product $V^{\frac{1}{2}}\otimes V^{\frac{1}{2}}$ of the auxiliary spaces only.

Similarly, for $\sigma\in S_N$ we define $\mathcal{U}_\sigma^{\xi}(x;\mathbf{t})=
\mathcal{U}_{\sigma}^{\xi,\underline{\ell}}(x;\mathbf{t})$ by
\begin{equation} \label{Udefn}
\begin{split}
\mathcal{U}_\sigma^{\xi}(x;\mathbf{t}):=& T_\sigma(x;\mathbf{t})^{-1}K^{\xi}(x)^{-1}T_\sigma(-x;\mathbf{t})\\
=&\left(\begin{matrix} \mathcal{A}_\sigma^{\xi}(x;\mathbf{t}) &
\mathcal{B}_\sigma^{\xi}(x;\mathbf{t})\\
\mathcal{C}_\sigma^{\xi}(x;\mathbf{t}) & \mathcal{D}_\sigma^{\xi}(x;\mathbf{t})
\end{matrix}\right)
\end{split}
\end{equation}
as a linear operator on $V^{\frac{1}{2}}\otimes M^{\underline{\ell}}$ (here $K^\xi(x)^{-1}$ only acts on the auxiliary space component of the tensor product). 
Then $\mathcal{U}^\xi(x;\mathbf{t}):=\mathcal{U}_\mathrm{e}^{\xi}(x;\mathbf{t})$ is the boundary monodromy operator \cite{Sk} associated to the $K$-operator $K^\xi$. 
The operators $\mathcal{U}^{\xi}_\sigma(x;\mathbf{t})$ satisfy the commutation relations
\begin{equation}\label{RU}
\begin{split}
R_{00^\prime}(y-x)\mathcal{U}^\xi_{\sigma,0}(x;\mathbf{t})&R_{00^\prime}(-x-y)\mathcal{U}^\xi_{\sigma,0^\prime}(y;\mathbf{t})=\\
=&\mathcal{U}^\xi_{\sigma,0^\prime}(y;\mathbf{t})R_{00^\prime}(-x-y)\mathcal{U}^\xi_{\sigma,0}(x;\mathbf{t})R_{00^\prime}(y-x)
\end{split}
\end{equation}
as linear operators on $V^{\frac{1}{2}}\otimes V^{\frac{1}{2}}\otimes M^{\underline{\ell}}$ with the same notational conventions as for \eqref{RTT}.
One of the consequences of these commutation relations is the commutativity of the operators $\mathcal{B}^\xi_\sigma$:
\[\lbrack \mathcal{B}_\sigma^{\xi}(x;\mathbf{t}),\mathcal{B}_\sigma^{\xi}(y;\mathbf{t})\rbrack=0.\]

\begin{rema}
Boundary transfer operators were constructed in \cite{Sk} in the context of quantum integrable models with boundaries.
In the present context the boundary transfer operator is the linear operator on $M^{\underline{\ell}}$ defined as
\begin{equation*}
\begin{split}
\mathcal{T}^{\xi_+,\xi_-}(x;\mathbf{t}):=&\textup{Tr}_{V^{\frac{1}{2}}}\bigl(K^{\xi_+}(x-\eta)\mathcal{U}^{\xi_-}(x;\mathbf{t})\bigr)\\
=&\mathcal{A}^{\xi_-}(x;\mathbf{t})+\frac{\sinh(\xi_+-x+\eta)}{\sinh(\xi_++x-\eta)}\mathcal{D}^{\xi_-}(x;\mathbf{t}),
\end{split}
\end{equation*}
where $\xi_+,\xi_-\in\mathbb{C}$. 
It is a commuting family of operators:
\[
\lbrack \mathcal{T}^{\xi_+,\xi_-}(x;\mathbf{t}),\mathcal{T}^{\xi_+,\xi_-}(y;\mathbf{t})\rbrack=0.
\]
In a similar way one can define boundary transfer operators acting on the same state space $M^{\underline{\ell}}$ but 
involving higher-spin representations $V^k$ ($k \in \frac{1}{2} \Z_{\geq 0}$) in the auxiliary space, 
similar to the situation for periodic boundary conditions (see for example, the lectures \cite{R-LH}). 
We will describe their properties in a separate publication.
\end{rema}

\subsection{The pseudo-vacuum and the Bethe vectors}
We write
\begin{equation*}
L^{\ell}(x)=
\left(\begin{matrix}
A^{\ell}(x) & B^{\ell}(x)\\
C^{\ell}(x) & D^{\ell}(x)
\end{matrix}\right)
\end{equation*}
with respect to the ordered basis $(v_1^{\frac{1}{2}},v_2^{\frac{1}{2}})$ of the auxiliary space. 
The matrix coefficients are linear operators on $M^{\ell}$. 
Explicitly they are given by
\begin{equation}\label{actionABCD}
\begin{split}
A^\ell(x)m_n^\ell&=\frac{\sinh(x+(\frac{3}{2}+\ell-n)\eta)}{\sinh(x+(\frac{1}{2}+\ell)\eta)}m_n^\ell,\\
B^\ell(x)m_n^\ell&=\frac{\sinh(\eta)}{\sinh(x+(\frac{1}{2}+\ell)\eta)}\mathrm{e}^{(-\ell-\frac{1}{2}+n)\eta}m_{n+1}^{\ell},\\
C^\ell(x)m_n^\ell&=\frac{\sinh((n-1)\eta)\sinh((2\ell+2-n)\eta)}{\sinh(\eta)\sinh(x+(\frac{1}{2}+\ell)\eta)}
\mathrm{e}^{(\ell+\frac{3}{2}-n)\eta}m_{n-1}^\ell,\\
D^\ell(x)m_n^\ell&=\frac{\sinh(x+(-\frac{1}{2}-\ell+n)\eta)}{\sinh(x+(\frac{1}{2}+\ell)\eta)}m_n^\ell,
\end{split}
\end{equation}
where $m_{0}^\ell$ should be read as zero.
Note that
\begin{gather*}
A^{\ell}(x)m_1^\ell=m_1^\ell, \qquad D^\ell(x)m_1^\ell=\vartheta^\ell(-x)m_1^\ell, \qquad  C^{\ell}(x)m_1^\ell=0,\\
\mathcal{R}^{k\ell}(x)(m_1^k\otimes m_1^\ell)=m_1^k\otimes m_1^\ell, \qquad  \mathcal{K}^{\xi,\ell}(x)m_1^\ell=m_1^\ell.
\end{gather*}
Here $\vartheta^\ell(x)$ is defined in \eqref{varthetal}.
Set
\begin{equation}\label{pseudovacuum}
\Omega:=m_1^{\ell_1}\otimes m_1^{\ell_2}\otimes\cdots\otimes m_1^{\ell_N}\in M^{\underline{\ell}}.
\end{equation}
Note that
\begin{equation}\label{eigenvalues}
A_\sigma(x;\mathbf{t})\Omega=\Omega,
\qquad D_\sigma(x;\mathbf{t})\Omega=\Bigl(\prod_{r=1}^N\vartheta^{\ell_r}(t_r-x)\Bigr)
\Omega
\end{equation}
for all $\sigma\in S_N$. 
The vector $\Omega$ will play the role of the pseudo-vacuum vector, from which off-shell Bethe vectors are generated by repeatedly applying operators $\mathcal B^\xi(x_i;\mathbf{t})$,
cf. \cite{Sk}.

For convenience, to construct our solutions to the boundary qKZ equations we will use a different normalization for $\mathcal{B}_\sigma^{\xi}(x;\mathbf{t})$:
\[
\overline{\mathcal{B}}_\sigma^{\xi}(x;\mathbf{t}):=\Bigl(\prod_{r=1}^N
\frac{\sinh(x-t_r-\ell_r\eta)}{\sinh(x-t_r+\ell_r\eta)}\Bigr)
\frac{\sinh(\xi-x-\frac{\eta}{2})\sinh(2x)}
{\sinh(2x+\eta)}\mathcal{B}_\sigma^{\xi}(x+
\frac{\eta}{2};\mathbf{t}). 
\]
The change from $\mathcal{B}$ to $\overline{\mathcal{B}}$ does not affect the commutativity:
\[\lbrack \overline{\mathcal{B}}_\sigma^{\xi}(x;\mathbf{t}),\overline{\mathcal{B}}_\sigma^{\xi}(y;\mathbf{t})\rbrack=0. \]
Hence, the following operator is well-defined for all $\mathbf{x}=(x_1,\ldots,x_S)$ with $S\in\mathbb{Z}_{\geq 0}$:
\[
\overline{\mathcal{B}}_{\sigma}^{\xi,(S)}(\mathbf{x};\mathbf{t}):=\prod_{j=1}^S\overline{\mathcal{B}}^{\xi}_\sigma(x_j;\mathbf{t}).
\]
We will write $\overline{\mathcal{B}}^\xi(x;\mathbf{t}):=
\overline{\mathcal{B}}_{e}^{\xi}(x;\mathbf{t})$ and
$\overline{\mathcal{B}}^{\xi,(S)}(\mathbf{x};\mathbf{t}):=
\overline{\mathcal{B}}_{e}^{\xi,(S)}(\mathbf{x};\mathbf{t})$
when $\sigma=e$ is the identity element of $S_N$.
The associated off-shell Bethe vectors are the vectors $\overline{\mathcal{B}}^{\xi,(S)}(\mathbf{x};\mathbf{t})  \Omega \in M^{\underline{\ell}}$.

\section{Jackson integral solutions of the boundary qKZ equations} \label{sec:mr}
We recall the notion of mero-uniformly convergent sums for scalar-valued functions (cf. \cite{Ru}), which can be extended to $M^{\underline \ell}$-valued functions in an obvious manner.
\begin{defi}
Let $\mathcal{C}\subset\C^M$ be a discrete subset and $w(\mathbf{x};\mathbf{t})$ ($\mathbf{x}\in\mathcal{C}$) a weight function with values depending meromorphically on $\mathbf{t}\in\C^N$.
Suppose that for all $\mathbf{t}_0\in\C^N$, there exists an open neighbourhood $U_{\mathbf{t}_0}\subset\C^N$ of $\mathbf{t}_0$ and a nonzero holomorphic function $v_{\mathbf{t}_0}$ on $U_{\mathbf{t}_0}$ such that
\begin{enumerate}
\item $v_{\mathbf{t}_0}(\mathbf{t})w(\mathbf{x};\mathbf{t})$ is holomorphic in $\mathbf{t}\in U_{\mathbf{t}_0}$ for all $\mathbf{x}\in\mathcal{C}$,
\item the sum $\sum_{\mathbf{x}\in\mathcal{C}}v_{\mathbf{t}_0}(\mathbf{t}) w(\mathbf{x};\mathbf{t})$ is absolutely and uniformly convergent for $\mathbf{t}\in U_{\mathbf{t}_0}$.
\end{enumerate}
Then there exists a unique meromorphic function $f(\mathbf{t})$ in $\mathbf{t}\in\C^N$ satisfying
\[ v_{\mathbf{t}_0}(\mathbf{t})f(\mathbf{t})=\sum_{\mathbf{x}\in\mathcal{C}} v_{\mathbf{t}_0}(\mathbf{t})w(\mathbf{x};\mathbf{t}) \]
for $\mathbf{t}\in U_{\mathbf{t}_0}$ and $\mathbf{t}_0\in\C^N$.
We will write
\[ f(\mathbf{t})=\sum_{\mathbf{x}\in\mathcal{C}}w(\mathbf{x};\mathbf{t}) \]
and we will say that the sum converges mero-uniformly.
\end{defi}

We are now in a position to present our main theorem.
For a meromorphic function $h$ of one variable, write $h(x\pm y)=h(x+y)h(x-y)$. 
\begin{thm}\label{mr}
Let $\xi_+,\xi_-\in\mathbb{C}$ and let $g_{\xi_+,\xi_-}(x)$, $h(x)$ and $F^{\ell}(x)$ be meromorphic functions in $x\in\mathbb{C}$ satisfying the functional equations
\begin{equation*}
\begin{split}
g_{\xi_+,\xi_-}(x+\tau)&= \frac{\sinh(\xi_--x-\frac{\eta}{2})\sinh(\xi_+-x-\frac{\tau}{2}-\frac{\eta}{2})}{\sinh(\xi_-+x+\tau-\frac{\eta}{2}) \sinh(\xi_++x+\frac{\tau}{2}-\frac{\eta}{2})}g_{\xi_+,\xi_-}(x),\\
h(x+\tau)&=\frac{\sinh(x+\tau)\sinh(x+\eta)}{\sinh(x)\sinh(x+\tau-\eta)}h(x),\\
F^\ell(x+\tau)&=\frac{\sinh(x+\tau-\ell\eta)}{\sinh(x+\tau+\ell\eta)}F^\ell(x).
\end{split}
\end{equation*}
Fix generic $\mathbf{x}_0\in\mathbb{C}^S$ and suppose that the $M^{\underline{\ell}}$-valued sum
\begin{equation*}
\begin{split}
f_S^{\underline{\ell}}(\mathbf{t}):=\sum_{\mathbf{x}\in\mathbf{x}_0+\tau\mathbb{Z}^S}
\Bigl(\prod_{i=1}^Sg_{\xi_+,\xi_-}(x_i)\Bigr)
&\Bigl(\prod_{1\leq i<j\leq S}h(x_i\pm x_j)\Bigr)\\
&\times\Bigl(\prod_{r=1}^N\prod_{i=1}^SF^{\ell_r}(t_r\pm x_i)\Bigr)
\overline{\mathcal{B}}^{\xi_-,(S)}(\mathbf{x};\mathbf{t})\Omega
\end{split}
\end{equation*}
converges mero-uniformly in $\mathbf{t}\in\mathbb{C}^N$.
Then $f_S^{\underline{\ell}}$ is a solution of the boundary qKZ equations \eqref{ReflqKZ}.
\end{thm}

Theorem \ref{mr} generalizes the main result of \cite{RSV} from $2$-dimensional representations of quantum $\mathfrak{sl}_2$  to arbitrary Verma modules.
The proof of Theorem \ref{mr} follows roughly the line of reasoning of the spin-$\frac{1}{2}$ case \cite{RSV}, although considerably more technical difficulties need to be overcome. The proof is given in Section \ref{sec:proof}. 

We now make the solutions concrete. 
We set $q:=\mathrm{e}^\tau$ and we assume that $\Re(\tau)<0$, so that $|q|<1$.
Solutions $g_{\xi_+,\xi_-}, h$ and $F^\ell$ of the resulting functional relations can now be expressed in terms of $q$-Gamma functions or, equivalently, in terms of $q$-shifted factorials
\[ \bigl(x;q\bigr)_{\infty}:=\prod_{i=0}^{\infty}(1-q^ix). \]
We write $\bigl(x_1,\ldots,x_s;q\bigr)_{\infty}:= \prod_{i=1}^s\bigl(x_i;q\bigr)_{\infty}$ for products of $q$-shifted factorials. 
As solutions of the functional equations we take
\begin{equation}\label{wghtexplicit}
\begin{split}
g_{\xi_+,\xi_-}(x)&=\mathrm{e}^{(\frac{2(\xi_-+\xi_+-\eta)}{\tau}+1)x}
\frac{\bigl(q^2 \mathrm{e}^{2(x+\xi_-)-\eta},q \mathrm{e}^{2(x+\xi_+)-\eta};q^2\bigr)_{\infty}}
{\bigl(\mathrm{e}^{2(x-\xi_-)+\eta},q\mathrm{e}^{2(x-\xi_+)+\eta};q^2\bigr)_{\infty}},\\
h(x)&=\mathrm{e}^{-\frac{2\eta x}{\tau}}(1-\mathrm{e}^{2x})\frac{
\bigl(q^2\mathrm{e}^{2(x-\eta)};q^2\bigr)_{\infty}}
{\bigl(\mathrm{e}^{2(x+\eta)};q^2\bigr)_{\infty}},\\
F^\ell(x)&=\mathrm{e}^{\frac{2\ell\eta x}{\tau}}
\frac{\bigl(q^2\mathrm{e}^{2(x+\ell\eta)};q^2\bigr)_{\infty}}
{\bigl(q^2\mathrm{e}^{2(x-\ell\eta)};q^2\bigr)_{\infty}}.
\end{split}
\end{equation}
With these choices for the solutions of the functional equations and the assumption that $\Re(\tau)<0$, it is readily established  (cf. \cite[Subsections 3.4 and 3.5]{RSV}) that the solution $f^{\underline \ell}_S(\mathbf t)$ defined in Theorem \ref{mr} converges mero-uniformly in $\mathbf{t}\in \mathbb{C}^N$ for generic $\mathbf{x}_0\in\mathbb{C}^S$ when $\Re(\eta)\geq 0$ and
\begin{equation} \label{convergencecondition}
\Re\bigl(2\xi_++2\xi_-+2\bigl(2\sum_{r=1}^N\ell_r-1\bigr)\eta+\tau\bigr)<0.
\end{equation}

\section{Proof of the main result}\label{sec:proof}

\subsection{Preliminary steps}
Let $S_N$ be the symmetric group in $N$ letters and $\sigma\in S_N$.
We view
\begin{equation}\label{Lhat}
L^{\ell_{\sigma(1)}}(x-t_{\sigma(1)})L^{\ell_{\sigma(2)}}(x-t_{\sigma(2)})
\cdots L^{\ell_{\sigma(N-1)}}(x-t_{\sigma(N-1)})
\end{equation}
as a linear operator on $V^{\frac{1}{2}}\otimes M^{\underline{\ell}}$ acting trivially on the tensor component
of $M^{\underline{\ell}}$ labelled by $\sigma(N)$. Write
\begin{equation*}
\left(\begin{matrix} \widehat{A}_\sigma(x;\mathbf{t}) & \widehat{B}_\sigma(x;\mathbf{t})\\
\widehat{C}_\sigma(x;\mathbf{t}) & \widehat{D}_\sigma(x;\mathbf{t})
\end{matrix}\right)
\end{equation*}
for the operator \eqref{Lhat}, written as a matrix with respect to the ordered basis $(v_1^{\frac{1}{2}},v_2^{\frac{1}{2}})$ of $V^{\frac{1}{2}}$. 
The operators $\widehat{A}_\sigma(x;\mathbf{t}),\ldots,\widehat{D}_\sigma(x;\mathbf{t})$ act on $M^{\underline{\ell}}$. 
They act trivially on the $\sigma(N)$-th tensor component of $M^{\underline{\ell}}$ and do not depend on $t_{\sigma(N)}$.

For $\sigma\in S_N$, $J\subseteq\{1,\ldots,S\}$ and $\bm{\epsilon}\in\{\pm\}^S$ we write
\begin{equation*}
\begin{split}
\mathcal{Y}_{\sigma}^{\xi,\bm{\epsilon},J}& (\mathbf{x};\mathbf{t}):=
\Bigl(\prod_{i=1}^S \epsilon_i\sinh(\xi-\epsilon_ix_i-\frac{\eta}{2})
\prod_{r=1}^N\frac{\sinh(\epsilon_ix_i-t_r-\ell_r\eta)}
{\sinh(\epsilon_ix_i-t_r+\ell_r\eta)}\Bigr)\\
&\times\Bigl(\prod_{1\leq i<j\leq S}\frac{\sinh(\epsilon_ix_i+\epsilon_jx_j+\eta)}
{\sinh(\epsilon_ix_i+\epsilon_jx_j)}\Bigr)
Y_\sigma^J((-\epsilon_1x_1-\frac{\eta}{2},\ldots,-\epsilon_Sx_S-\frac{\eta}{2});
\mathbf{t})
\end{split}
\end{equation*}
with
\begin{equation*}
Y_\sigma^{J}(\mathbf{x};\mathbf{t}) :=\Bigl(\prod_{i\in J}\frac{\sinh(x_i-t_{\sigma(N)}+(\frac{1}{2}-\ell_{\sigma(N)})\eta)}{\sinh(x_i-t_{\sigma(N)}+(\frac{1}{2}+\ell_{\sigma(N)})\eta)}\Bigr) \hspace{-2mm} \prod_{(i,j)\in J\times J^c} \hspace{-2mm} \frac{\sinh(x_i-x_j+\eta)}{\sinh(x_i-x_j)}
\end{equation*}
and $J^c:=\{1,\ldots,S\}\setminus J$ (empty products are equal to one). 
Similarly to the spin-$\frac{1}{2}$ case (see \cite[Cor. 4.3]{RSV}) we have the explicit expression
\[
\begin{split}
{\overline{\mathcal{B}}}_{\sigma}^{\xi,(S)} (\mathbf{x};\mathbf{t})\Omega &=
\sum_{\bm{\epsilon}\in\{\pm 1\}^S}\sum_{J\subseteq\{1,\ldots,S\}}
\mathcal{Y}_{\sigma}^{\xi,\bm{\epsilon},J}(\mathbf{x};\mathbf{t}) \\
&\; \times \Bigl(\prod_{j\in J^c}B^{\ell_{\sigma(N)}}(-\epsilon_jx_j-\frac{\eta}{2}-
t_{\sigma(N)})\Bigr) \Bigl(\prod_{i\in J}\widehat{B}_\sigma(-\epsilon_ix_i
-\frac{\eta}{2};\mathbf{t})\Bigr)\Omega
\end{split}
\] 
of the Bethe vector (see \cite[Cor. 4.3]{RSV}). 
For $r \in \{1,\ldots,N-1\}$ write $s_r \in S_N$ for the simple neighbouring transposition $r \leftrightarrow r+1$.
In \cite[Lemma 5.4]{RSV} the condition that the function $f_S^{\underline \ell}(\bm t)$ with $\ell = (\tfrac{1}{2},\cdots,\tfrac{1}{2})$ satisfies the boundary qKZ equations is re-written as a system of equations involving the weight functions $\mathcal{Y}_{\sigma}^{\xi, \bm{\epsilon},J}$ where $\sigma = s_r \cdots s_{N-1}$ for some $r \in \{1,\ldots,N\}$.
This directly generalizes to the following result in the current higher-spin context.
\begin{lem}\label{technicalreform}
Provided mero-uniform convergence,
\[
f_S^{\underline{\ell}}(\mathbf{t}):=\sum_{\mathbf{x}\in\mathbf{x}_0+\tau\mathbb{Z}^S}
w^{(S)}(\mathbf{x};\mathbf{t};\xi_+,\xi_-)
\overline{\mathcal{B}}^{\xi_-,(S)}(\mathbf{x};\mathbf{t})\Omega
\]
satisfies the boundary qKZ equations \eqref{ReflqKZ} iff
\begin{equation}\label{lhs}
\begin{split}
\sum_{\mathbf{x},\bm{\epsilon},J}&w^{(S)}(\mathbf{x};\mathbf{t};\xi_+,\xi_-)
\Bigl(\prod_{i=1}^S\frac{\sinh(\pm x_i+t_r+\ell_r\eta)}{\sinh(\pm x_i+t_r-\ell_r\eta)} \Bigr)
\mathcal{Y}_{s_r\cdots s_{N-1}}^{\xi_-,\bm{\epsilon},J}(\mathbf{x};
e_r\mathbf{t}) \\
& \times \mathcal{K}^{\xi_+,\ell_r}(t_r+\frac{\tau}{2})\\
&\times
\Bigl(\prod_{j\in J^c}B^{\ell_r}(-\epsilon_jx_j-\frac{\eta}{2}+t_r)\Bigr)
\Bigl(\prod_{i\in J}\widehat{B}_{s_r\cdots s_{N-1}}(-\epsilon_ix_i-\frac{\eta}{2};
\mathbf{t})\Bigr)\Omega
\end{split}
\end{equation}
equals
\begin{equation}\label{rhs}
\begin{split}
&\hspace{-4mm} \sum_{\mathbf{x},\bm{\epsilon},J}w^{(S)}(\mathbf{x};\mathbf{t}+\tau \bm e_r;
\xi_+,\xi_-)\mathcal{Y}_{s_r\cdots s_{N-1}}^{\xi_-,\bm{\epsilon},J}(\mathbf{x};
\mathbf{t}+\tau \bm e_r)\\
& \times\Bigl(\prod_{j\in J^c}B^{\ell_r}(-\epsilon_jx_j-\frac{\eta}{2}-t_r-
\tau)\Bigr)\Bigl(\prod_{i\in J}\widehat{B}_{s_r\cdots s_{N-1}}
(-\epsilon_ix_i-\frac{\eta}{2};\mathbf{t})\Bigr)\Omega
\end{split}
\end{equation}
for $r=1,\ldots,N$,
where the summations are over $\mathbf{x}\in\mathbf{x}_0+\tau\mathbb{Z}^S$,
$\bm{\epsilon}\in\{\pm\}^S$ and over subsets $J\subseteq\{1,\ldots,S\}$ (recall that $J^c=\{1,\dots, S\}\backslash J$).
\end{lem}
We fix $S\geq 1$ and suppress it from the notations. For $d\in\{0,\ldots,S\}$ set $\mathcal{L}^{(d)}_r(\mathbf{t})$ and
$\mathcal{R}^{(d)}_r(\mathbf{t})$ for \eqref{lhs} and \eqref{rhs} respectively,
with the sums running over $\mathbf{x}\in\mathbf{x}_0+\tau\mathbb{Z}^S$,
$\bm{\epsilon}\in\{\pm\}^S$ and over subsets $J\subseteq\{1,\ldots,S\}$
of cardinality $S-d$. The strategy of the proof of Theorem \ref{mr} is to determine
sufficients conditions on the weight function $w^{(S)}(\mathbf{x};\mathbf{t};\xi_+,\xi_-)$ so that
\begin{equation}\label{TODOstep1}
\mathcal{L}^{(d)}_r(\mathbf{t})=\mathcal{R}^{(d)}_r(\mathbf{t})
\end{equation}
for all $d\in\{0,\ldots,S\}$ and $r\in\{1,\ldots,N\}$. We will call $d$ the {\it depth}.
\begin{rema}
In the study \cite{RSV} of Jackson integral solutions for the spin-$\frac{1}{2}$ representations
the terms $\mathcal{L}^{(d)}_r(\mathbf{t})$ and $\mathcal{R}^{(d)}_r(\mathbf{t})$ are automatically
zero if $d\geq 2$, cf. \cite[Rem. 5.5]{RSV}. When $M^{\ell_s}$ are highest weight modules with $\ell_s\in \C$
we have to deal with the terms $\mathcal{L}^{(d)}_r(\mathbf{t})$ and $\mathcal{R}^{(d)}_r(\mathbf{t})$
for any depth $d\in\{0,\ldots,S\}$.
\end{rema}

\subsection{Depth zero}
Completely analogous to the spin-$\frac{1}{2}$ case (see \cite[\S 5.1]{RSV}) we have the following result.

\begin{lem}\label{highestlemma}
Suppose that
\[
w^{(S)}(\mathbf{x};\mathbf{t};\xi_+,\xi_-)=
\Bigl(\prod_{r=1}^N\prod_{i=1}^SF^{\ell_r}(t_r\pm x_i)\Bigr)
G_{\xi_+,\xi_-}^{(S)}(\mathbf{x})
\]
with $G_{\xi_+,\xi_-}^{(S)}(\mathbf{x})$ independent of $\mathbf{t}$. If
\[
F^{\ell_r}(x+\tau)=
\frac{\sinh(x+\tau-\ell_r\eta)}{\sinh(x+\tau+\ell_r\eta)}F^{\ell_r}(x)
\]
for $r=1,\ldots,N$
then, provided mero-uniform convergence,
\begin{equation}\label{hwequal}
\mathcal{L}_r^{(0)}(\mathbf{t};\xi_+,\xi_-)=
\mathcal{R}_r^{(0)}(\mathbf{t};\xi_+,\xi_-)
\end{equation}
for $r=1,\ldots,N$.
\end{lem}
In the remainder of the section we assume that the weight function $w^{(S)}(\mathbf{x};\mathbf{t};\xi_+,\xi_-)$ is of the form as specified in
Lemma \ref{highestlemma}.
\subsection{The remaining depths}
We have the setup that
\[
f_S^{\underline{\ell}}(\mathbf{t})=\sum_{\mathbf{x}\in\mathbf{x}_0+\tau\mathbb{Z}^S}
w^{(S)}(\mathbf{x};\mathbf{t};\xi_+,\xi_-)
\overline{\mathcal{B}}^{\xi_-,(S)}(\mathbf{x};\mathbf{t})\Omega
\]
with the sum converging mero-uniformly in $\mathbf{t}\in\mathbb{C}^N$
and with weight function of the form
\begin{equation}\label{form}
w^{(S)}(\mathbf{x};\mathbf{t};\xi_+,\xi_-)=
\Bigl(\prod_{r=1}^N\prod_{i=1}^SF^{\ell_r}(t_r\pm x_i)\Bigr)
G_{\xi_+,\xi_-}^{(S)}(\mathbf{x})
\end{equation}
with $G_{\xi_+,\xi_-}^{(S)}(\mathbf{x})$ independent of $\mathbf{t}$ and with
the $F^{\ell}$
satisfying
\begin{equation}\label{form2}
F^{\ell}(x+\tau)=
\frac{\sinh(x+\tau-\ell\eta)}
{\sinh(x+\tau+\ell\eta)}F^\ell(x).
\end{equation}
We are now going to show that conditions on the weight factor $G_{\xi_+,\xi_-}^{(S)}(\mathbf{x})$ as stated in
Theorem \ref{mr} imply that \eqref{TODOstep1} is valid for $d\in\{1,\ldots,S\}$ and $r\in\{1,\ldots,N\}$. Combined with Lemma \ref{highestlemma}
and Lemma \ref{technicalreform}, this will complete the proof of Theorem \ref{mr}.

Since the $\xi_\pm$ are fixed throughout this subsection, we will
suppress $\xi_\pm$ from the notations; in particular, we
write $w^{(S)}(\mathbf{x};\mathbf{t})$ for
$w^{(S)}(\mathbf{x};\mathbf{t};\xi_+,\xi_-)$. We also suppress $S\in\mathbb{Z}_{\geq 1}$ from the notations.

If $J\subseteq\{1,\ldots,S\}$, $\epsilon\in\{\pm\}^S$ and
$\mathbf{x}\in\mathbf{x}_0+\tau\mathbb{Z}^S$ then we write
$\mathbf{x}_J:=(x_j)_{j\in J}$ and $\bm{\epsilon}_J:=(\epsilon_j)_{j\in J}$.
Conversely, for given $\bm{\epsilon}_J$ and
$\bm{\epsilon}_{J^c}$ the associated $S$-tuple
of signs will be denoted by $\bm{\epsilon}$
(and similarly for $\mathbf{x}$).

It is convenient to define the following weights.
\begin{defi}
For $r\in\{1,\ldots,N\}$, $\epsilon\in\{\pm\}^S$ and a subset
$J\subseteq\{1,\ldots,S\}$ we write
\[
m^{\bm{\epsilon},J}_r(\mathbf{x};\mathbf{t}):=
\Bigl(\prod_{i=1}^S\frac{\sinh(\pm x_i-t_r+\ell_r\eta)}
{\sinh(\pm x_i-t_r-\ell_r\eta)}\Bigr)
\frac{\mathcal{Y}_{s_r\cdots s_{N-1}}^{\xi_-,\bm{\epsilon},J}
(\mathbf{x};\mathbf{t})}
{\prod_{j\in J^c}\sinh(-\epsilon_jx_j-t_r+\ell_r\eta)}
\]
for $\mathbf{x}\in\mathbf{x}_0+\tau\mathbb{Z}^S$.
\end{defi}
It follows by a straightforward computation that
\begin{equation}\label{m}
\begin{split}
m_r^{\bm{\epsilon},J}(\mathbf{x};\mathbf{t})&=\Bigl(\prod_{j\in J^c}\Bigl(
\epsilon_j\frac{\sinh(\xi_--\epsilon_jx_j-\frac{\eta}{2})}
{\sinh(-t_r-\epsilon_jx_j-\ell_r\eta)}\prod_{\stackrel{s=1}{s\not=r}}^N
\frac{\sinh(t_s-\epsilon_jx_j+\ell_s\eta)}
{\sinh(t_s-\epsilon_jx_j-\ell_s\eta)}\Bigr)\Bigr)\\
&\times\Bigl(\prod_{(i,j)\in J\times J^c}\frac{\sinh(\epsilon_jx_j\pm x_i+\eta)}
{\sinh(\epsilon_jx_j\pm x_i)}\Bigr)
\Bigl(\prod_{\stackrel{i,i^\prime\in J:}{i<i^\prime}}
\frac{\sinh(\epsilon_ix_i+\epsilon_{i^\prime}x_{i^\prime}+\eta)}
{\sinh(\epsilon_ix_i+\epsilon_{i^\prime}x_{i^\prime})}\Bigr)\\
&\times\Bigl(\prod_{\stackrel{j,j^\prime\in J^c:}
{j<j^\prime}}\frac{\sinh(\epsilon_jx_j+\epsilon_{j^\prime}x_{j^\prime}+\eta)}
{\sinh(\epsilon_jx_j+\epsilon_{j^\prime}x_{j^\prime})}\Bigr)\\
&\times\prod_{i\in J}\Bigl(
\epsilon_i\sinh(\xi_--\epsilon_ix_i-\frac{\eta}{2})
\prod_{\stackrel{s=1}{s\not=r}}^N\frac{\sinh(t_s-\epsilon_ix_i+\ell_s\eta)}
{\sinh(t_s-\epsilon_ix_i-\ell_s\eta)}\Bigr).
\end{split}
\end{equation}
\begin{lem}
Fix $d\in\{1,\ldots,S\}$ and $r\in\{1,\ldots,N\}$.
Suppose that
for all subsets $J\subseteq\{1,\ldots,S\}$ of cardinality $S-d$ and for
all $\mathbf{x}_J$ and $\bm{\epsilon}_J$,
\begin{equation*}
\begin{split}
C_{d+1}^{\ell_r}(t_r+\frac{\tau}{2};\xi_+)
&\sum_{\mathbf{x}_{J^c},\bm{\epsilon}_{J^c}}
w^{(S)}(\mathbf{x};\mathbf{t};\xi_+,\xi_-)m_r^{\bm{\epsilon},J}(\mathbf{x};
e_r\mathbf{t})=\\
&\qquad\qquad=\sum_{\mathbf{x}_{J^c},\bm{\epsilon}_{J^c}}
w^{(S)}(\mathbf{x};\mathbf{t};\xi_+,\xi_-)m_r^{\bm{\epsilon},J}(\mathbf{x};
\mathbf{t}+\tau \bm e_r).
\end{split}
\end{equation*}
Then $\mathcal{L}_r^{(d)}(\mathbf{t})=\mathcal{R}_r^{(d)}(\mathbf{t})$.
\end{lem}
\begin{proof}
Recall that
\[
\mathcal{K}^{\xi_+,\ell_r}(t_r+\frac{\tau}{2})m_{d+1}^{\ell_r}=C_{d+1}^{\ell_r}(t_r+\frac{\tau}{2};\xi_+)
m_{d+1}^{\ell_r},
\]
see Definition \ref{Kmatrixdef}.
Since
\[
\Bigl(\prod_{j\in J^c}B^{\ell_r}(-\epsilon_jx_j-\frac{\eta}{2}+u)\Bigr)
m_1^{\ell_r}=
\frac{\sinh^d(\eta)\mathrm{e}^{\eta(\frac{d^2}{2}-\ell_rd)}}{\prod_{j\in J^c}\sinh(-\epsilon_jx_j+u+\ell_r\eta)}
m_{d+1}^{\ell_r}
\]
by \eqref{actionABCD}
we thus have
\begin{equation*}
\begin{split}
&\mathcal{K}^{\xi_+,\ell_r}(t_r+\frac{\tau}{2})
\Bigl(\prod_{j\in J^c}B^{\ell_r}(-\epsilon_jx_j-\frac{\eta}{2}+t_r)\Bigr) m_1^{\ell_r}= \\
&\qquad\qquad = \frac{\sinh^d(\eta)C_{d+1}^{\ell_r}(t_r+\frac{\tau}{2};\xi_+)
\mathrm{e}^{\eta(\frac{d^2}{2}-\ell_rd)}}{\prod_{j\in J^c}\sinh(-\epsilon_jx_j+t_r+\ell_r\eta)}m_{d+1}^{\ell_r},\\
& \Bigl(\prod_{j\in J^c}B^{\ell_r}(-\epsilon_jx_j-\frac{\eta}{2}-\tau-t_r)\Bigr)m_1^{\ell_r}=\\
&\qquad\qquad=\frac{\sinh^d(\eta)\mathrm{e}^{\eta(\frac{d^2}{2}-\ell_rd)}}
{\prod_{j\in J^c}\sinh(-\epsilon_jx_j-\tau-t_r+\ell_r\eta)}m_{d+1}^{\ell_r}.
\end{split}
\end{equation*}
Taking the expressions \eqref{lhs} and \eqref{rhs}
for $\mathcal{L}_r^{(m)}(\mathbf{t})$
and $\mathcal{R}_r^{(m)}(\mathbf{t})$
into account we conclude that
$\mathcal{L}_r^{(d)}(\mathbf{t})=\mathcal{R}_r^{(d)}(\mathbf{t})$ if
for all subsets $J\subseteq\{1,\ldots,S\}$ of cardinality $S-d$ and for
all $\mathbf{x}_J$ and $\bm{\epsilon}_J$,
\begin{equation*}
\begin{split}
&C_{d+1}^{\ell_r}(t_r+\frac{\tau}{2};\xi_+)
\sum_{\mathbf{x}_{J^c},\bm{\epsilon}_{J^c}}w^{(S)}(\mathbf{x};\mathbf{t};
\xi_+,\xi_-)m_r^{\bm{\epsilon},J}(\mathbf{x};e_r\mathbf{t})=\\
& \qquad =\sum_{\mathbf{x}_{J^c},\bm{\epsilon}_{J^c}}w^{(S)}(\mathbf{x};
\mathbf{t}+\tau \bm e_r;\xi_+,\xi_-) \\
& \qquad \qquad \times \Bigl(
\prod_{i=1}^S\frac{\sinh(\pm x_i-t_r-\tau-\ell_r\eta)}
{\sinh(\pm x_i-t_r-\tau+\ell_r\eta)}\Bigr)
m_r^{\bm{\epsilon},J}(\mathbf{x};\mathbf{t}+\tau \bm e_r).
\end{split}
\end{equation*}
The lemma now follows from the fact that
\[
w^{(S)}(\mathbf{x};\mathbf{t}+\tau \bm e_r;\xi_+,\xi_-)=\Bigl(\prod_{i=1}^S\frac{\sinh(\pm x_i-t_r-\tau+\ell_r\eta)}
{\sinh(\pm x_i-t_r-\tau-\ell_r\eta)}\Bigr)w^{(S)}(\mathbf{x};\mathbf{t};\xi_+,\xi_-),
\]
which is a direct consequence of the specific form \eqref{form}, \eqref{form2} of the weight function $w^{(S)}(\mathbf{x};\mathbf{t})$.
\end{proof}
In the remainder of this subsection we fix $d\in\{1,\ldots,S\}$,
$r\in\{1,\ldots,N\}$, a subset $J\subseteq\{1,\ldots,S\}$ of cardinality
$S-d$, as well as $\mathbf{x}_J$ and $\bm{\epsilon}_J$, which we all suppress from the notations. Set for
$\bm{\epsilon}_{J^c}\in\{\pm\}^d$,
\begin{equation*}
\begin{split}
\Lambda_{r,\bm{\epsilon}_{J^c}}(\mathbf{t})&:=\sum_{\mathbf{x}_{J^c}}
w^{(S)}(\mathbf{x};\mathbf{t};\xi_+,\xi_-)m_r^{\bm{\epsilon},J}(\mathbf{x};
e_r \mathbf t),\\
\Upsilon_{r,\bm{\epsilon}_{J^c}}(\mathbf{t})&:=\sum_{\mathbf{x}_{J^c}}
w^{(S)}(\mathbf{x};\mathbf{t};\xi_+,\xi_-)m_r^{\bm{\epsilon},J}(\mathbf{x};
\mathbf{t}+\tau \bm e_r).
\end{split}
\end{equation*}
In view of the previous lemma, the desired identity \eqref{TODOstep1} follows if
\begin{equation}\label{TODOrewrite}
C_{d+1}^{\ell_r}(t_r+\frac{\tau}{2};\xi_+)\sum_{\bm{\epsilon}_{J^c}\in\{\pm\}^d}
\Lambda_{r,\bm{\epsilon}_{J^c}}(\mathbf{t})=
\sum_{\bm{\epsilon}_{J^c}\in\{\pm\}^d}
\Upsilon_{r,\bm{\epsilon}_{J^c}}(\mathbf{t}).
\end{equation}
We write $m_r(\mathbf{x}_{J^c};\mathbf{t})$ for
$m_r^{\bm{\epsilon},J}(\mathbf{x};\mathbf{t})$ with
$\bm{\epsilon}_{J^c}$ the $d$-tuple $(-,-,\cdots,-)$
of minus signs,
\begin{equation}\label{mbasis}
\begin{split}
&m_r(\mathbf{x}_{J^c};\mathbf{t})=(-1)^d
\Bigl(\prod_{j\in J^c}\Bigl(\frac{\sinh(\xi_-+x_j-\frac{\eta}{2})}
{\sinh(-t_r+x_j-\ell_r\eta)}\prod_{\stackrel{s=1}{s\not=r}}^N
\frac{\sinh(t_s+x_j+\ell_s\eta)}{\sinh(t_s+x_j-\ell_s\eta)}\Bigr)\Bigr) \hspace{-15mm}\\
&\qquad\times\Bigl(\prod_{(i,j)\in J\times J^c}\frac{\sinh(x_j\pm x_i-\eta)}
{\sinh(x_j\pm x_i)}\Bigr)\Bigl(\prod_{\stackrel{i,i^\prime\in J:}
{i<i^\prime}}\frac{\sinh(\epsilon_ix_i+\epsilon_{i^\prime}x_{i^\prime}+\eta)}
{\sinh(\epsilon_ix_i+\epsilon_{i^\prime}x_{i^\prime})}\Bigr)\\
&\qquad\times\Bigl(\prod_{\stackrel{j,j^\prime\in J^c:}{j<j^\prime}}
\frac{\sinh(x_j+x_{j^\prime}-\eta)}{\sinh(x_j+x_{j^\prime})}\Bigr) \\
&\qquad \times \Bigl(\prod_{i\in J}\Bigl(\epsilon_i
\sinh(\xi_--\epsilon_ix_i-\frac{\eta}{2})
\prod_{\stackrel{s=1}{s\not=r}}^N\frac{\sinh(t_s-\epsilon_ix_i+\ell_s\eta)}
{\sinh(t_s-\epsilon_ix_i-\ell_s\eta)}\Bigr)\Bigr).
\end{split}
\end{equation}
\begin{lem}\label{reduction}
Suppose that for all $i\in\{1,\ldots,S\}$,
\begin{equation}\label{Gstep}
\begin{split}
& G_{\xi_+,\xi_-}(\mathbf{x}-\tau \bm e_i)=\frac{\sinh(\xi_-+x_i-\frac{\eta}{2})
\sinh(\xi_++x_i-\frac{\tau}{2}-\frac{\eta}{2})}
{\sinh(\xi_--x_i+\tau-\frac{\eta}{2})\sinh(\xi_+-x_i+\frac{\tau}{2}-
\frac{\eta}{2})}\\
& \qquad\qquad \times\Bigl(\prod_{\stackrel{i^\prime=1}{i^\prime\not=i}}^S
\frac{\sinh(x_i\pm x_{i^\prime}-\tau)\sinh(x_i\pm x_{i^\prime}-\eta)}
{\sinh(x_i\pm x_{i^\prime}-\tau+\eta)\sinh(x_i\pm x_{i^\prime})}
\Bigr)G_{\xi_+,\xi_-}(\mathbf{x}).
\end{split}
\end{equation}
Then
\begin{equation*}
\begin{split}
\Lambda_{r,\bm{\epsilon}_{J^c}}(\mathbf{t})&=(-1)^{\#J_+^c}
\sum_{\mathbf{x}_{J^c}}w^{(S)}(\mathbf{x};\mathbf{t};\xi_+,\xi_-)
q_{J_+^c}(\mathbf{x}_{J^c};t_r)m_r(\mathbf{x}_{J^c};e_r\mathbf{t}),\\
\Upsilon_{r,\bm{\epsilon}_{J^c}}(\mathbf{t})&=(-1)^{\#J_+^c}
\sum_{\mathbf{x}_{J^c}}w^{(S)}(\mathbf{x};\mathbf{t};\xi_+,\xi_-)
q_{J_+^c}(\mathbf{x}_{J^c};-t_r-\tau)m_r(\mathbf{x}_{J^c};\mathbf{t}+\tau \bm e_r)
\end{split}
\end{equation*}
with $J_+^c:=\{j\in J^c \, | \, \epsilon_j=+\}$, $J_-^c:=J^c\setminus J_+^c$ and
\begin{equation*}
\begin{split}
q_{J_+^c}(\mathbf{x}_{J^c};t_r):=&\Bigl(
\prod_{\stackrel{j,j^\prime\in J_+^c:}{j<j^\prime}}\frac{\sinh(x_j+x_{j^\prime}-
\tau-\eta)}{\sinh(x_j+x_{j^\prime}-\tau+\eta)}\Bigr)\\
\times&\Bigl(
\prod_{j\in J_-^c}\prod_{j^\prime\in J_+^c}
\frac{\sinh(x_{j^\prime}-x_j-\eta)\sinh(x_{j^\prime}+x_j-\tau)}
{\sinh(x_{j^\prime}-x_j)\sinh(x_{j^\prime}+x_j-\tau+\eta)}\Bigr)\\
\times&
\Bigl(\prod_{j\in J_+^c}\frac{\sinh(\xi_++x_j-\frac{\tau}{2}-\frac{\eta}{2})
\sinh(t_r+x_j+\ell_r\eta)}{\sinh(\xi_+-x_j+\frac{\tau}{2}-\frac{\eta}{2})
\sinh(t_r-x_j+\tau+\ell_r\eta)}\Bigr).
\end{split}
\end{equation*}
\end{lem}
\begin{proof}
The formula for $\Lambda_{r,\bm{\epsilon}_{J^c}}(\mathbf{t})$ is correct if $\bm{\epsilon}_{J^c}$ is the $d$-tuple $(-,-\cdots,-)$ of minus signs since
$q_\emptyset(\mathbf{x}_{J^c};t_r)=1$ (empty products are equal to one by convention).

Fix $\bm{\epsilon}_{J^c}\in\{\pm\}^d$ and $I \subset J^c_+$.
Write $\bm{\epsilon}_{J^c}^{I,-}$ for the $d$-tuple of signs obtained from $\bm{\epsilon}_{J^c}$ by replacing $\epsilon_i=+$ by $-$ for all $i\in I$.
Similarly, we write $\bm{\epsilon}^{I,-}$ for the $S$-tupe of signs obtained from $\bm{\epsilon}$ by replacing $\epsilon_i=+$ by $-$ for all $i\in I$.

Fix $k\in J_+^c$ and rewrite $\Lambda_{r,\bm{\epsilon}_{J^c}}(\mathbf{t})$ as
\[
\Lambda_{r,\bm{\epsilon}_{J^c}}(\mathbf{t})=
\sum_{\mathbf{x}_{J^c}}w^{(S)}(\mathbf{x}-\tau \bm e_k;\mathbf{t};\xi_+,\xi_-)
m_r^{\bm{\epsilon},J}(\mathbf{x}-\tau \bm e_k;e_r\mathbf{t}).
\]
By the assumptions on $w^{(S)}(\mathbf{x};\mathbf{t})$ we have
\[
w^{(S)}(\mathbf{x}-\tau \bm e_k;\mathbf{t};\xi_+,\xi_-)=
\beta_k(\mathbf{x};\mathbf{t})w^{(S)}(\mathbf{x};\mathbf{t};\xi_+,\xi_-)
\]
with
\begin{equation*}
\begin{split}
\beta_k(\mathbf{x};\mathbf{t}):=&
\Bigl(\prod_{s=1}^N\frac{\sinh(t_s+x_k+\ell_s\eta)\sinh(t_s-x_k+\tau-\ell_s\eta)}
{\sinh(t_s+x_k-\ell_s\eta)\sinh(t_s-x_k+\tau+\ell_s\eta)}\Bigr)\\
\times&\frac{\sinh(\xi_-+x_k-\frac{\eta}{2})\sinh(\xi_++x_k-\frac{\tau}{2}-
\frac{\eta}{2})}{\sinh(\xi_--x_k+\tau-\frac{\eta}{2})
\sinh(\xi_+-x_k+\frac{\tau}{2}-\frac{\eta}{2})}\\
\times&\Bigl(\prod_{\stackrel{k^\prime=1}{k^\prime\not=k}}^S
\frac{\sinh(x_k\pm x_{k^\prime}-\tau)\sinh(x_k\pm x_{k^\prime}-\eta)}
{\sinh(x_k\pm x_{k^\prime}-\tau+\eta)\sinh(x_k\pm x_{k^\prime})}\Bigr).
\end{split}
\end{equation*}
In addition, by a direct computation using \eqref{m},
\[
\beta_k(\mathbf{x};\mathbf{t})m_r^{\bm{\epsilon},J}(\mathbf{x}-\tau \bm e_k;e_r\mathbf{t})=
-\gamma_k^{\bm{\epsilon}_{J^c}}(\mathbf{x}_{J^c};t_r) m_r^{\bm{\epsilon}^{\{k\},-},J}(\mathbf{x};e_r\mathbf{t})
\]
with
\begin{equation*}
\begin{split}
\gamma_k^{\bm{\epsilon}_{J^c}}(\mathbf{x}_{J^c};t_r):=&
\Bigl(\prod_{j\in J^c\setminus \{k\}}\frac{\sinh(x_k+\epsilon_jx_j-\eta)\sinh(x_k-\epsilon_jx_j-\tau)}{\sinh(x_k+\epsilon_jx_j)\sinh(x_k-\epsilon_jx_j-\tau+\eta)}\Bigr) \\
 \times & \frac{\sinh(\xi_++x_k-\frac{\tau}{2}-\frac{\eta}{2})\sinh(t_r+x_k+\ell_r\eta)}{\sinh(\xi_+-x_k+\frac{\tau}{2}-\frac{\eta}{2})\sinh(t_r-x_k+\tau+\ell_r\eta)}.
\end{split}
\end{equation*}
Hence
\[
\Lambda_{r,\bm{\epsilon}_{J^c}}(\mathbf{t})= -\sum_{\mathbf{x}_{J^c}}w^{(S)}(\mathbf{x};\mathbf{t};\xi_+,\xi_-)
\gamma_k^{\bm{\epsilon}_{J^c}}(\mathbf{x}_{J^c};t_r) m_r^{\bm{\epsilon}^{\{k\},-},J}(\mathbf{x};e_r\mathbf{t}).
\]
This in particular proves the desired expression of
$\Lambda_{r,\bm{\epsilon}_{J^c}}(\mathbf{t})$ if $\epsilon_k=+$ and $\epsilon_j=-$ for $j\in J^c\setminus\{k\}$.

The formula for arbitrary $\bm{\epsilon}_{J^c}\in\{\pm\}^d$ follows by an induction argument with respect to
$\# J^c_+$ using the following observation.
For a subset $I\subseteq J_+^c$ set
\begin{equation*}
\begin{split}
\widetilde{q}_I(\mathbf{x}_{J^c};t_r):=&
\Bigl(\prod_{\stackrel{i,i^\prime\in I:}{i<i^\prime}}
\frac{\sinh(x_i+x_{i^\prime}-\tau-\eta)}{\sinh(x_i+x_{i^\prime}-\tau+\eta)}\Bigr)\\
\times&\Bigl(
\prod_{i\in I}\frac{\sinh(\xi_++x_i-\frac{\tau}{2}-\frac{\eta}{2})
\sinh(t_r+x_i+\ell_r\eta)}{\sinh(\xi_+-x_i+\frac{\tau}{2}-\frac{\eta}{2})
\sinh(t_r-x_i+\tau+\ell_r\eta)}\Bigr)\\
\times&\prod_{(i,j)\in I\times J^c\setminus I}
\frac{\sinh(x_i+\epsilon_jx_j-\eta)\sinh(x_i-\epsilon_jx_j-\tau)}
{\sinh(x_i+\epsilon_jx_j)\sinh(x_i-\epsilon_jx_j-\tau+\eta)}.
\end{split}
\end{equation*}
Then $\widetilde{q}_\emptyset(\mathbf{x}_{J^c};t_r)=1$, $\widetilde{q}_{J_+^c}(\mathbf{x}_{J^c};t_r)=q_{J_+^c}(\mathbf{x}_{J^c};t_r)$ and
for a subset $I\subset J_+^c$ and $k\in J_+^c\setminus I$,
\[
\frac{\widetilde{q}_{I\cup\{k\}}(\mathbf{x}_{J^c};t_r)}{\widetilde{q}_I(\mathbf{x}_{J^c}-\tau \bm e_k;t_r)} = \gamma_k^{\bm{\epsilon}^{I,-}_{J^c}}(\mathbf{x}_{J^c};t_r) .
\]

The alternative expression for $\Upsilon_{r,\bm{\epsilon}_{J^c}}(\mathbf{t})$ follows from a
similar computation, now using the observation that for $k\in J_+^c$,
\[
\beta_k(\mathbf{x};\mathbf{t})m_r^{\bm{\epsilon},J}(\mathbf{x}-\tau \bm e_k;\mathbf{t}+\tau \bm e_r)=
-\gamma_k^{\bm{\epsilon}_{J^c}}(\mathbf{x}_{J^c};-t_r-\tau)m_r^{\bm{\epsilon}^{\{k\},-},J}(\mathbf{x};\mathbf{t}+\tau \bm e_r). \qedhere
\]
\end{proof}
Note that \eqref{Gstep} is satisfied if
\[
G_{\xi_+,\xi_-}(\mathbf{x})=
\Bigl(\prod_{i=1}^Sg_{\xi_+,\xi_-}(x_i)\Bigr)\prod_{1\leq i<i^\prime\leq S}
h(x_i\pm x_{i^\prime})
\]
with $g_{\xi_+,\xi_-}$ and $h$ as in Theorem \ref{mr}.

By the explicit expression \eqref{mbasis} of $m_r(\mathbf{x}_{J^c};\mathbf{t})$
we have
\begin{equation}\label{mrel}
\begin{split}
\widetilde{m}_r(\mathbf{x}_{J^c};\mathbf{t}):=&
m_r(\mathbf{x}_{J^c};e_r\mathbf{t})
\prod_{j\in J^c}\sinh(t_r+x_j-\ell_r\eta)\\
=&m_r(\mathbf{x}_{J^c};\mathbf{t}+\tau \bm e_r)
\prod_{j\in J^c}\sinh(-t_r-\tau+x_j-\ell_r\eta).
\end{split}
\end{equation}
Combined with Lemma \ref{reduction}, it follows that
\eqref{TODOrewrite} is equivalent to
\begin{equation*}
\begin{split}
&\sum_{\mathbf{x}_{J^c}}w^{(S)}(\mathbf{x};\mathbf{t};\xi_+,\xi_-)\widetilde{m}_r(\mathbf{x}_{J^c};\mathbf{t}) C_{d+1}^{\ell_r}(t_r+\frac{\tau}{2};\xi_+) \\
& \hspace{50mm} \times \sum_{\bm{\epsilon}_{J^c}\in\{\pm\}^d} \frac{(-1)^{\#J_+^c}q_{J_+^c}(\mathbf{x}_{J^c};t_r)}{\prod_{j\in J^c}\sinh(t_r+x_j-\ell_r\eta)}=\\
&\,=\sum_{\mathbf{x}_{J^c}}w^{(S)}(\mathbf{x};\mathbf{t};\xi_+,\xi_-) \widetilde{m}_r(\mathbf{x}_{J^c};\mathbf{t}) \sum_{\bm{\epsilon}_{J^c}\in\{\pm\}^d} \frac{(-1)^{\#J_+^c}q_{J_+^c}(\mathbf{x}_{J^c};-t_r-\tau)}{\prod_{j\in J^c}\sinh(-t_r-\tau+x_j-\ell_r\eta)}.
\end{split}
\end{equation*}
Substituting the explicit expression \eqref{Cell}
of  $C_{n}^{\ell}(x;\xi)$, this is a direct consequence of the following lemma.
\begin{lem}\label{trigident}
Let $J\subseteq\{1,\ldots,S\}$ be a subset of cardinality $S-d$ and
$\bm{\epsilon}_{J^c}\in\{\pm\}^d$.
Then the finite sum
\[
\Bigl(\prod_{n=1}^d\sinh(\xi_+-t_r-\frac{\tau}{2}+(\ell_r+\frac{1}{2}-n)\eta)
\Bigr)\sum_{\bm{\epsilon}_{J^c}\in\{\pm\}^d}
\frac{(-1)^{\#J_+^c}q_{J_+^c}(\mathbf{x}_{J^c};t_r)}
{\prod_{j\in J^c}\sinh(t_r+x_j-\ell_r\eta)}
\]
is invariant under the exchange of $t_r$ by $-t_r-\tau$.
\end{lem}
The proof of the lemma is given in the next subsection. It completes
the proof of the main theorem (Theorem \ref{mr}).
\subsection{Proof of Lemma \ref{trigident}}\label{AppenA}
Let $J\subseteq\{1,\ldots,S\}$ be a subset of cardinality $S-d$ and
$\bm{\epsilon}_{J^c}\in\{\pm\}^d$.
Choose an identification of the fixed subset $J^c$ of cardinality $d$ with $\{1,\ldots,d\}$.
The choice of signs $\bm{\epsilon}_{J^c}\in\{\pm\}^d$ then is identified
with choosing a subset $I\subseteq\{1,\ldots,d\}$ by the rule
\[
I:=\{i\in\{1,\ldots,d\}\, | \, \epsilon_i=+\}.
\]
Write $\xi = \xi_+ - \frac{\eta}{2}$ and $\mathbf{x}=(x_1,\ldots,x_d)$.
Then the statement in Lemma \ref{trigident} is easily seen to be equivalent to the claim that
\begin{equation}  \begin{aligned}
F(\mathbf x;t)&:=
\Bigl( \prod_{i = 1}^d \frac{\sinh(\xi-t-\frac{\tau}{2} +(\ell + 1 - i)\eta)}{\sinh(t+x_i-\ell \eta)} \Bigr)\\
& \times\sum_{I \subseteq \{ 1,\ldots, d\}} \Biggl\{ (-1)^{\#I}
 \Bigl( \prod_{i,j \in I \atop i<j} \frac{\sinh(x_i+x_j-\tau-\eta)}{\sinh(x_i+x_j-\tau+\eta)} \Bigr) \\
& \qquad\quad \times  \biggl(  \prod_{i \in I} \frac{\sinh(\xi+x_i-\frac{\tau}{2})\sinh(t+x_i+\ell \eta)}{\sinh(\xi-x_i+\frac{\tau}{2})\sinh(t-x_i+\tau+\ell \eta)} \biggr) \\
& \qquad\quad \times \prod_{(i,j) \in I \times I^\mathrm{c}} \frac{\sinh(x_i-x_j-\eta)\sinh(x_i+x_j-\tau)}{\sinh(x_i-x_j)\sinh(x_i+x_j-\tau+\eta)}\biggr) \Biggr\}
\end{aligned} \end{equation}
satisfies
\begin{equation} \label{toprove}
F(\mathbf x; -t-\tau) = F(\mathbf x; t).
\end{equation}
By substituting $x_i \to x_i+\frac{\tau}{2}$ ($i=1,\ldots,d$) and $t \to t-\frac{\tau}{2}$ and clearing denominators in \eqref{toprove}, we obtain a trigonometric polynomial identity independent of $\tau$.
More precisely, for $i \in \{1,\ldots,d\}$ and $I \subseteq \{1,\ldots,d\}$ write $\epsilon_i^{(I)} = +$ if $i \in I$ and $\epsilon_i^{(I)} = -$ if $i \not\in I$; also, write $x_i^{(I)} = x_i - \epsilon_i^{(I)} \frac{\eta}{2}$.
For $I \subseteq \{1,\ldots,d\}$ we define
\begin{align*}
Q_I(\mathbf x;t) &:=  (-1)^{\#I} \Bigl( \prod_{i =1}^d \sinh(\xi+\epsilon_i^{(I)} x_i) \sinh(t +  \epsilon_i^{(I)}x_i + \ell \eta) \Bigr) \\
& \hspace{45mm} \prod_{1\leq i<j\leq d}\sinh(x_i^{(I)} \pm x_j^{(I)})
\end{align*}
and write
\[
V(\mathbf x;t) := \Bigl( \prod_{i=1}^d \sinh(\xi-t+(\ell - i + 1)\eta) \Bigr) \sum_{I \subseteq \{ 1,\ldots, d\}}Q_I(\mathbf x;t).
\]
Then \eqref{toprove} is equivalent to
\begin{equation} \label{Pidentity}
V(\mathbf x;t) = V(\mathbf x;-t).
\end{equation}
The identity \eqref{Pidentity} is a direct consequence of the following multivariate generalization of the trigonometric identity \eqref{coeffcond4}.
\begin{lem}\label{multivariable}
We have
\begin{equation} \label{Pequation}
\begin{split}
 \sum_{I \subseteq \{ 1,\ldots, d\}}Q_I(\mathbf x;t) &=\Bigl(\prod_{1\leq i<j\leq d}\sinh(x_i \pm x_j) \Bigr)\prod_{i=1}^d \sinh(2x_i)\\
 &\times (-1)^d\prod_{i=1}^d\sinh(\xi + t + (\ell - i + 1)\eta).
 \end{split}
 \end{equation}
\end{lem}
\begin{proof}
Write $\mathbb{V}(\mathbf{x};t)$ for the left-hand side of \eqref{Pequation}.
It is easy to see that
\[
\mathbb{V}(\mathbf{x};t)\in\mathbb{C}[\mathrm{e}^{\pm 2x_1},\ldots,\mathrm{e}^{\pm 2x_d}],
\]
since each term $Q_I(\mathbf{x};t)$ is a Laurent polynomial in $\mathrm{e}^{2x_1},\ldots,\mathrm{e}^{2x_d}$. 
We now first show that $\mathbb{V}(\mathbf{x};t)$ is anti-invariant with respect to the natural action of the Weyl group $W$ of type C$_d$ on $\mathbb{C}[\mathrm{e}^{\pm 2x_1},\ldots,\mathrm{e}^{\pm 2x_d}]$.

Let $W = \langle s_1, \ldots, s_d \rangle$ be the Weyl group of type C$_d$, with the simple reflections $s_i$ ($i=1,\ldots,d$) acting
on $\C^d$ by permutations and sign flips: for $1\leq i<d$ the simple reflection
$s_i$ acts on $(x_1,\ldots,x_d) \in \C^d$ by permuting $x_i$ and $x_{i+1}$, and
$s_d$ acts by sending $x_d$ to $-x_d$. The Weyl group $W$ also acts on the power set of $\{1,\ldots,d\}$ by
\begin{gather*}
s_i I = \begin{cases} (I \setminus \{ i \} ) \cup \{ i+1 \}, & \text{if } i \in I, \, i+1 \not \in I, \\
(I \setminus \{ i+1 \} ) \cup \{ i \}, & \text{if } i \not \in I, \, i+1 \in I, \\
I, & \text{otherwise}
\end{cases}
\end{gather*}
for $1\leq i<d$, and
\begin{gather*}
s_d I = \begin{cases} I \setminus \{d \}, & \text{if } d \in I, \\ I \cup \{ d \}, & \text{if } d \not \in I. \end{cases}
\end{gather*}
Note that the action of $W$ on the power set of $\{1,\ldots,d\}$ is transitive, and that the stabilizer subgroup of the empty set $\emptyset$
is equal to the symmetric group $S_d:=\langle s_1,\ldots,s_{d-1}\rangle$ in $d$ letters.

By a direct computation we obtain the invariance property
\begin{equation} \label{Qtransformation}
Q_I(w\mathbf x;t) = (-1)^{l(w)}Q_{w^{-1}I}(\mathbf x;t), \qquad w\in W,
\end{equation}
where $l(w)$ is the length of $w\in W$. It follows that
\[
\mathbb{V}(\mathbf{x};t)=\frac{1}{d!}\sum_{w\in W}(-1)^{l(w)}Q_\emptyset(w^{-1}\mathbf{x};t),
\]
in particular $\mathbb{V}(\mathbf{x};t)\in\mathbb{C}[\mathrm{e}^{\pm 2x_1},\ldots,\mathrm{e}^{\pm 2x_d}]$ is $W$-anti-invariant. 
Thus
\begin{equation}\label{stepWeyl0}
\mathbb{V}(\mathbf{x};t)=Z(\mathbf{x};t) \delta(\mathbf{x})
\end{equation}
with the Weyl denominator
\[
\delta(\mathbf{x}):=\Bigl(\prod_{1\leq i<j\leq d}\sinh(x_i \pm x_j) \Bigr)\prod_{i=1}^d \sinh(2x_i)
\]
and with $Z(\mathbf{x};t)\in\mathbb{C}[\mathrm{e}^{\pm 2x_1},\ldots,\mathrm{e}^{2x_d}]$ $W$-invariant. A standard argument comparing degrees on both sides of
\eqref{stepWeyl0}
shows that $Z(\mathbf{x};t)$ is independent of $\mathbf{x}$. So
\begin{equation}\label{stepWeyl}
\mathbb{V}(\mathbf{x};t)=Z(t)\delta(\mathbf{x})
\end{equation}
for some constant $Z(t)$. We compute $Z(t)$ by evaluating both sides of \eqref{stepWeyl} in
\[
\mathbf{y}:=(-\xi+(d-1)\eta,-\xi+(d-2)\eta,\ldots,-\xi).
\]
By the explicit expression
\[
Q_\emptyset(\mathbf{x};t)=\! \Bigl(\prod_{i=1}^d\sinh(\xi-x_i)\sinh(t-x_i+\ell\eta)\Bigr) \hspace{-2mm}
\prod_{1\leq i<j\leq d} \hspace{-2mm} \sinh(x_i-x_j)\sinh(x_i+x_j+\eta)
\]
it follows that $Q_\emptyset(w^{-1}\mathbf{y};t)=0$ for $w\in W$ unless $w\in S_d$. Hence
\[
\mathbb{V}(\mathbf{y};t)=\frac{1}{d!}\sum_{w\in S_d}(-1)^{l(w)}Q_\emptyset(w^{-1}\mathbf{y};t)=
\frac{1}{d!}\sum_{w\in S_d}Q_{w\emptyset}(\mathbf{y};t)=Q_\emptyset(\mathbf{y};t),
\]
and consequently
\[
Z(t)=\frac{Q_\emptyset(\mathbf{y};t)}{\delta(\mathbf{y})}=
(-1)^d\prod_{i=1}^d\sinh(\xi + t + (\ell - i + 1)\eta),
\]
where the last equality follows from a straightforward computation.
\end{proof}


\section{Fusion for the boundary qKZ equations and their solutions.} \label{sec:fusion2}

In this section we will show that, for $\underline{\ell} \in \frac{1}{2} \Z_{\geq 0}^N$, the solutions $f_S^{\underline{\ell}}(\mathbf{t})$ exhibited in Theorem \ref{mr} can be directly obtained using a fusion process from the spin-half solution $\bigl(\textup{pr}^{\frac{1}{2}}\bigr)^{\otimes N}\bigl(f_S^{(\frac{1}{2},\ldots,\frac{1}{2})}(\mathbf{t})\bigr)$
constructed before in \cite{RSV}.
Moreover, as we will see, arbitrary solutions of the boundary qKZ equations \eqref{ReflqKZ} in $M^{(\ell_1,\ldots,\ell_{s-1})} \otimes V^k \otimes V^\frac{1}{2} \otimes M^{(\ell_{s+1},\ldots,\ell_N)}$ can be naturally fused to obtain solutions in $M^{(\ell_1,\ldots,\ell_{s-1})} \otimes V^{k+\frac{1}{2}} \otimes M^{(\ell_{s+1},\ldots,\ell_N)}$.

\subsection{Notations}
In this section, we will slightly abuse notation when considering operators acting on a ``mixed'' $N$-fold tensor product made up of finite- and infinite-dimensional modules $V^k$ ($k \in \frac{1}{2} \Z_{\geq 0}$) and $M^\ell$ ($\ell \in \C$).
For example, if $\ell_s \in \frac{1}{2}\Z_{\geq 0}$, there is a unique linear operator $\tilde \Xi^{\underline{\ell}}_r(\mathbf t;\xi_+,\xi_-;\tau)$ on $M^{(\ell_1,\ldots,\ell_{s-1})} \otimes V^{\ell_s} \otimes M^{(\ell_{s+1},\ldots,\ell_N)}$ determined by
\begin{align*}
& \tilde \Xi^{\underline{\ell}}_r(\mathbf t;\xi_+,\xi_-;\tau) \Bigl( \textup{Id}_{M^{(\ell_1,\ldots,\ell_{s-1})}} \otimes \textup{pr}^{\ell_s} \otimes \textup{Id}_{M^{(\ell_{s+1},\ldots,\ell_N)}}  \Bigr) = \\
&\qquad = \Bigl( \textup{Id}_{M^{(\ell_1,\ldots,\ell_{s-1})}} \otimes \textup{pr}^{\ell_s} \otimes \textup{Id}_{M^{(\ell_{s+1},\ldots,\ell_N)}}  \Bigr) \Xi^{\underline{\ell}}_r(\mathbf t;\xi_+,\xi_-;\tau);
\end{align*}
we will denote the resulting operator $\tilde\Xi^{\underline{\ell}}_r(\mathbf t;\xi_+,\xi_-;\tau)$ on $M^{(\ell_1,\ldots,\ell_{s-1})} \otimes V^{\ell_s} \otimes M^{(\ell_{s+1},\ldots,\ell_N)}$ simply by $\Xi^{\underline{\ell}}_r(\mathbf t;\xi_+,\xi_-;\tau)$ as long as it is clear from context which tensor component we have projected onto its finite-dimensional quotient.

We will use this mild abuse of notation also when discussing the operators $T^{\underline \ell}(x;\mathbf t)$, $\mathcal{U}^{\xi,\underline \ell}(x;\mathbf{t})$, 
$\mathcal{B}^{\xi,\underline \ell}(x;\mathbf{t})$, $\overline{\mathcal{B}}^{\xi,\underline \ell}(x;\mathbf{t})$ and $\overline{\mathcal{B}}^{\xi,(S),\underline \ell}(\mathbf{x};\mathbf{t})$.
Similarly, we will use the notations $\Omega^{\underline \ell}$ and $f^{\underline \ell}_S(\mathbf t)$ for those elements of $M^{(\ell_1,\ldots,\ell_{s-1})} \otimes V^{\ell_s} \otimes M^{(\ell_{s+1},\ldots,\ell_N)}$ that are actually given by $\textup{pr}^{\ell_s} \Omega^{\underline \ell}$ and $\textup{pr}^{\ell_s} f^{\underline \ell}_S(\mathbf t)$, respectively.

To fuse the boundary qKZ transport operators $\Xi_r^{\underline{\ell}}(\mathbf t) := \Xi_r^{\underline{\ell}}(\mathbf{t};\xi_+,\xi_-;\tau)$, it is convenient to use the injection
$j^k=P^{\frac{1}{2}k}\iota^k: V^{k+\frac{1}{2}}\hookrightarrow V^{k}\otimes V^{\frac{1}{2}}$ instead of $\iota^k$.  Let $k\in\frac{1}{2}\mathbb{Z}_{\geq 0}$ and $\ell\in\mathbb{C}$.
The following ``local'' fusion relations in terms of $j^k$ follow straightforwardly from Proposition \ref{FusionL} and \eqref{fusionk} respectively,
\begin{gather}
(j^k\otimes \textup{Id}_{M^\ell}) L^{k+\frac{1}{2},\ell}(x-y)=
L_{23}^{\frac{1}{2}\ell}(x-k\eta-y)L_{13}^{k\ell}(x+\frac{\eta}{2}-y)(j^k\otimes\textup{Id}_{M^\ell}), \label{Lfusionj} \\
j^k K^{k+\frac{1}{2}}(x) = K^{\frac{1}{2}}_2(x-k\eta)R^{k\frac{1}{2}}(2x-(k-\frac{1}{2})\eta) K_1^k(x+\frac{\eta}{2}) j^k. \label{Kfusionj}
\end{gather}
Furthermore, in a similar way as we derived Proposition \ref{FusionL} and \eqref{Lfusionj},
\begin{equation}
(j^k\otimes \textup{Id}_{M^\ell}) L^{k+\frac{1}{2},\ell}(x-y)=
L_{13}^{k\ell}(x-\frac{\eta}{2}-y)L_{23}^{\frac{1}{2}\ell}(x+k \eta-y)(j^k\otimes\textup{Id}_{M^\ell}). \label{Linvfusionj}
\end{equation}

Given $s=1,\ldots,N$ and $k\in\frac{1}{2}\mathbb{Z}_{\geq 0}$, denote
\begin{align*}
j^k_s &:= \textup{Id}_{M^{(\ell_1,\ldots,\ell_{s-1})}} \otimes j^k \otimes  \textup{Id}_{M^{(\ell_{s+1},\ldots,\ell_N)}},
\end{align*}
an injective map from $M^{(\ell_1,\ldots,\ell_{s-1})} \otimes V^{k+\frac{1}{2}} \otimes M^{(\ell_{s+1},\ldots,\ell_N)}$ to $M^{(\ell_1,\ldots,\ell_{s-1})} \otimes V^k \otimes V^\frac{1}{2} \otimes M^{(\ell_{s+1},\ldots,\ell_N)}$.

For the rest of this section, given $1 \leq s \leq N$ and $\underline \ell \in \C^N$ such that $\ell_s = k+\frac{1}{2}$ for $k \in \frac{1}{2} \Z_{\geq 0}$, we write
\begin{equation}\label{expanded}
\begin{split}
\underline{\ell}'&= (\ell_1,\ldots,\ell_{s-1},k,\frac{1}{2},\ell_{s+1},\ldots,\ell_N) \in \C^{N+1}, \\
\mathbf t' &= (t_1,\ldots,t_{s-1},t_s+\frac{\eta}{2},t_s-k \eta,t_{s+1},\ldots,t_N),
\end{split}
\end{equation}
while $\mathbf t=(t_1,\ldots,t_{s-1},t_s,t_{s+1},\ldots,t_N)$ and $\underline{\ell}=(\ell_1,\dots, \ell_N)$ with $\ell_s=k+\frac{1}{2}$.

\subsection{Fusion of transport operators}

\begin{prop}\label{transportoperatorfusion}
Let $1 \leq s \leq N$ and $\underline \ell  \in \C^N$ such that $\ell_s = k+\frac{1}{2}$ for $k \in \frac{1}{2} \Z_{\geq 0}$.
For $1 \leq r \leq N$ we have
\begin{equation} \label{Xifusion}
j^k_s \Xi_r^{\underline{\ell}}(\mathbf t) = \begin{cases} \Xi_r^{\underline{\ell}'}(\mathbf t') j^k_s, & r<s, \\
\Xi_{s+1}^{\underline{\ell}'}(\mathbf t' + \mathbf e_s \tau) \Xi_s^{\underline{\ell}'}(\mathbf t') j^k_s, & r=s, \\
\Xi_{r+1}^{\underline{\ell}'}(\mathbf t') j^k_s, & r>s, \end{cases}
\end{equation}
as linear operators $ M^{(\ell_1,\ldots,\ell_{s-1})} \otimes V^{\ell_s} \otimes M^{(\ell_{s+1},\ldots,\ell_N)} \to M^{(\ell_1,\ldots,\ell_{s-1})} \otimes V^{k} \otimes V^{\frac{1}{2}} \otimes M^{(\ell_{s+1},\ldots,\ell_N)}$.
\end{prop}
\begin{proof}
For the cases where $r \ne s$, simply by judiciously applying (\ref{Lfusionj}-\ref{Linvfusionj})
to the right-hand side of \eqref{Xifusion} (see \eqref{Atauj} for the definition of the transport operators).
For $r=s$, the product of factors in $\Xi_{s+1}^{\underline{\ell}'}(\mathbf t' + \mathbf e_s \tau) \Xi_s^{\underline{\ell}'}(\mathbf t')$ can first be simplified using unitarity of the $R$-operator and the RLL-relations \eqref{RLL}, yielding
\begin{align*}
& \Xi_{s+1}^{\underline{\ell}'}(\mathbf t' + \mathbf e_s \tau) \Xi_s^{\underline{\ell}'}(\mathbf t') = \Biggl( \prod_{j=s+1}^N L^{\frac{1}{2} \ell_j}(t_s - t_j + \tau - k \eta) L^{k \ell_j}(t_s - t_j + \tau + \frac{\eta}{2}) \Biggr) \\
& \qquad\qquad \times K^{\xi_+,\frac{1}{2}}(t_s + \frac{\tau}{2} - k \eta) R^{k \frac{1}{2}}(2(t_s + \frac{\tau}{2}) - (k-\frac{1}{2})\eta) K^{\xi_+,k}(t_s + \frac{\tau}{2} + \frac{\eta}{2}) \\
& \qquad\qquad \times \Biggl( \prod_{j=N \atop j \ne s}^1 L^{\frac{1}{2}\ell_j}(t_j + t_s - k \eta) L^{k\ell_j}(t_j + t_s + \frac{\eta}{2}) \Biggr)  \\
& \qquad\qquad \times K^{\xi_-,\frac{1}{2}}(t_s - k \eta) R^{\frac{1}{2} k}(2t_s - (k-\frac{1}{2})\eta) K^{\xi_-,k}(t_s + \frac{\eta}{2}) \\
& \qquad\qquad \times \Biggl( \prod_{j=1}^{s-1} L^{\frac{1}{2} \ell_j}(t_s - t_j - k \eta) L^{k \ell_j}(t_s - t_j + \frac{\eta}{2}) \Biggr),
\end{align*}
where the ordering of the products over $j$ is as prescribed.
Now applying (\ref{Lfusionj}-\ref{Kfusionj}) yields \eqref{Xifusion} for the case $r=s$.
\end{proof}
\subsection{Fusion of solutions}\label{fusionsubsection}

\begin{prop}
Let $1 \leq s \leq N$ and $\underline \ell  \in \C^N$ such that $\ell_s = k+\frac{1}{2}$ for $k \in \frac{1}{2} \Z_{\geq 0}$.
Suppose that $f:\mathbb{C}^{N+1}\rightarrow M^{(\ell_1,\ldots,\ell_{s-1})} \otimes V^k \otimes V^\frac{1}{2} \otimes M^{(\ell_{s+1},\ldots,\ell_N)}$
is a meromorphic solution of the boundary qKZ equations,
\begin{equation}\label{initialequation}
\Xi_r^{\underline{\ell}'}(\mathbf{z})f(\mathbf{z})=f(\mathbf{z}+\tau\mathbf{e}_r),\qquad 1\leq r\leq N+1,
\end{equation}
where $\underline{\ell}'$ is given by \eqref{expanded}. Suppose that $f$ restricts to a meromorphic function on the hyperplane
\[
H:=\{\mathbf{z}\in\mathbb{C}^{N+1} \,\, | \,\, z_s-z_{s+1}=(k+\frac{1}{2})\eta\, \}.
\]
Then there exists a unique meromorphic function
\[\textup{Fus}^{\underline{\ell}}_s(f): \mathbb{C}^N\rightarrow
M^{(\ell_1,\ldots,\ell_{s-1})} \otimes V^{k +\frac{1}{2}} \otimes M^{(\ell_{s+1},\ldots,\ell_N)}
\]
satisfying
\begin{equation} \label{solfusion}
j^k_s \textup{Fus}^{\underline{\ell}}_s(f)(\mathbf t) = f(\mathbf t'),
\end{equation}
with $\mathbf t'$ given by \eqref{expanded}. Furthermore, $\textup{Fus}^{\underline{\ell}}_s(f)$ is a meromorphic
solution of the boundary qKZ equations \eqref{ReflqKZ} with values in $M^{(\ell_1,\ldots,\ell_{s-1})} \otimes V^{k+\frac{1}{2}} \otimes M^{(\ell_{s+1},\ldots,\ell_N)}$,
\begin{equation}\label{fusqKZ}
\Xi_r^{\underline{\ell}}(\mathbf{t})\textup{Fus}_s^k(f)(\mathbf{t})=\textup{Fus}_s^k(f)(\mathbf{t}+\tau\mathbf{e}_r),\qquad 1\leq r\leq N.
\end{equation}
\end{prop}
\begin{proof}
It follows from \eqref{initialequation} with $r=s$ that $f(\mathbf{z})=\Xi_s^{\underline{\ell}'}(\mathbf{z}-\tau\mathbf{e}_s)f(\mathbf{z}-\tau\mathbf{e}_s)$.
By assumption the left-hand side restricts to a meromorphic vector valued function on $H$. By the explicit expressions \eqref{Atauj} for the transport operators,
the operator $\Xi_s^{\underline{\ell}'}(\mathbf{z}-\tau\mathbf{e}_s)$
restricts to a meromorphic operator valued function on $H$, and
\[ \Xi_s^{\underline{\ell}'}(\cdot\,-\tau\mathbf{e}_s)|_H=R^{k\frac{1}{2}}((k+\frac{1}{2})\eta)Z(\cdot)
\]
for some meromorphic operator valued function $Z$ on $H$. Hence $f|_H$ takes its values in the subspace
$\textup{Im}(R^{k\frac{1}{2}}((k+\frac{1}{2})\eta))$ of $M^{(\ell_1,\ldots,\ell_{s-1})}\otimes V^k\otimes V^{\frac{1}{2}}\otimes M^{(\ell_{s+1},\ldots,\ell_N)}$.
By Lemma \ref{link} we have $\textup{Im}(R^{k\frac{1}{2}}((k+\frac{1}{2})\eta))\subseteq\textup{Im}(j_s^k)$. Since $j_s^k$ is injective, we conclude that there exists
a unique meromorphic function
\[\textup{Fus}_s^{\underline{\ell}}(f): \mathbb{C}^N\rightarrow M^{(\ell_1,\ldots,\ell_{s-1})} \otimes V^{k +\frac{1}{2}} \otimes M^{(\ell_{s+1},\ldots,\ell_N)}
\]
satisfying \eqref{solfusion}.

It remains to show that $\textup{Fus}_s^{\underline{\ell}}(f)$ satisfies the boundary qKZ equations \eqref{fusqKZ}.
Since $j^k$ is an injection, it suffices to prove that, for $r=1,\ldots,N$,
\begin{equation}\label{TODOfusionsol}
j^k_s \Xi^{\underline{\ell}}_r(\mathbf t)  \textup{Fus}^{\underline{\ell}}_s(f)(\mathbf t) =  j^k_s \textup{Fus}^{\underline{\ell}}_s(f)(\mathbf t + \tau\mathbf e_r).
\end{equation}
For $r<s$ we have
\[
j^k_s \Xi^{\underline{\ell}}_r(\mathbf t) \textup{Fus}^{\underline{\ell}}_s(f)(\mathbf t) = \Xi_r^{\underline{\ell}'}(\mathbf t') f(\mathbf t') =  f(\mathbf t' + \tau\mathbf e_r) = j^k_s \textup{Fus}^{\underline{\ell}}_s(f)(\mathbf t + \tau\mathbf e_r) ,
\]
owing to \eqref{Xifusion}, \eqref{solfusion}, the boundary qKZ equations \eqref{initialequation} and \eqref{solfusion} again.
The case $r>s$ of \eqref{TODOfusionsol} is proven similarly.
Finally, for $r=s$ we have
\begin{equation*}
\begin{split}
j^k_s \Xi^{\underline{\ell}}_s(\mathbf t) \textup{Fus}^{\underline{\ell}}_s(f)(\mathbf t)
&= \Xi_{s+1}^{\underline{\ell}'}(\mathbf t' + \tau\mathbf e_s) \Xi_s^{\underline{\ell}'}(\mathbf t') f(\mathbf t')\\
&= \Xi_{s+1}^{\underline{\ell}'}(\mathbf t' + \tau\mathbf e_s)  f(\mathbf t' + \tau\mathbf e_s) \\
&= f(\mathbf t' + \tau\mathbf e_s  + \tau\mathbf e_{s+1}) =
j^k_s \textup{Fus}^{\underline{\ell}}_s(f)(\mathbf t + \tau\mathbf e_s),
\end{split}
\end{equation*}
where we have applied \eqref{Xifusion}, \eqref{solfusion}, \eqref{initialequation} twice, and finally \eqref{solfusion} again.
\end{proof}

\subsection{Fusion of the Jackson integral solutions}
The special Jackson integral solutions of the boundary qKZ equations (see Theorem \ref{mr})
are compatible with fusion in the following sense.
\begin{prop} \label{hypergeomsolfusion}
Let $1 \leq s \leq N$ and $\underline \ell\in\mathbb{C}^N$ such that $\ell_s = k+\frac{1}{2}$ with $k \in \frac{1}{2} \Z_{\geq 0}$.
Let $\ell^\prime\in\mathbb{C}^{N+1}$ be given by \eqref{expanded}. Let
\[
f_S^{\underline \ell}: \mathbb{C}^N\rightarrow M^{(\ell_1,\ldots,\ell_{s-1})}\otimes V^{k+\frac{1}{2}}\otimes
M^{(\ell_{s+1},\ldots,\ell_N)}
\]
and
\[
f_S^{\underline{\ell}'}: \mathbb{C}^{N+1}\rightarrow
M^{(\ell_1,\ldots,\ell_{s-1})}\otimes V^k\otimes V^{\frac{1}{2}}\otimes M^{(\ell_s,\ldots,\ell_N)}
\]
be the Jackson integral solutions of the boundary qKZ equations
as given in Theorem \ref{mr}, with $f_S^{\underline{\ell}}$ and $f_S^{\underline{\ell}'}$ having the same base point $\mathbf{x}_0\in\mathbb{C}^S$, the same weight factors $g_{\xi_+,\xi_-}$, $h$ and $F^{\ell_j}$ ($j\in\{1,\ldots,N\}\setminus\{s\}$) and with the remaining weight factors $F^{k+\frac{1}{2}}$, $F^k$ and $F^{\frac{1}{2}}$ satisfying the compatibility condition
\begin{equation}\label{compatibleF}
F^{k+\frac{1}{2}}(x)=F^k(x+\frac{\eta}{2})F^{\frac{1}{2}}(x-k\eta).
\end{equation}
Then
\[ f_S^{\underline \ell} = \textup{Fus}_s^{\underline \ell}(f_S^{\underline \ell'}).  \]
\end{prop}
\begin{rema}
Note that \eqref{compatibleF} is compatible with the difference equations that $F^\ell(x)$ satisfies (see Theorem \ref{mr}). Note furthermore that the explicit choice \eqref{wghtexplicit} of $F^\ell(x)$ ($\ell\in\mathbb{C}$) satisfies \eqref{compatibleF}.
\end{rema}
\begin{proof}
By virtue of the fusion formulae \eqref{Linvfusionj} and \eqref{Lfusionj}, we have (cf. \eqref{Tdefn})
\[ j^k_s T^{\underline \ell}(x;\mathbf t) = T^{\underline \ell'}(x;\mathbf t') j^k_s, \qquad  \qquad
j^k_s T^{\underline \ell}(x;\mathbf t)^{-1} = T^{\underline \ell'}(x;\mathbf t')^{-1} j^k_s,\]
where we use the notations \eqref{expanded}.
Hence, owing to \eqref{Udefn} we also have
\begin{equation} \label{Ufusion} 
j^k_s \mathcal U^{\xi,\underline \ell}(x;\mathbf t) = \mathcal U^{\xi,\underline \ell'}(x;\mathbf t') j^k_s.  
\end{equation}
The above three identities are as operators $V^{\frac{1}{2}} \otimes M^{(\ell_1,\ldots,\ell_{s-1})} \otimes V^{k+\frac{1}{2}} \otimes M^{(\ell_{s+1},\ldots,\ell_N)}$ $\to$ $V^{\frac{1}{2}} \otimes M^{(\ell_1,\ldots,\ell_{s-1})} \otimes V^k \otimes V^\frac{1}{2} \otimes M^{(\ell_{s+1},\ldots,\ell_N)}$.
Taking the appropriate matrix coefficients in \eqref{Ufusion} with respect to the auxiliary space,
we obtain
\[ j^k_s \mathcal B^{\xi,\underline \ell}(x;\mathbf t) = \mathcal B^{\xi,\underline \ell'}(x;\mathbf t') j^k_s \]
as operators $M^{(\ell_1,\ldots,\ell_{s-1})} \otimes V^{k+\frac{1}{2}} \otimes M^{(\ell_{s+1},\ldots,\ell_N)} \to M^{(\ell_1,\ldots,\ell_{s-1})} \otimes V^k \otimes V^\frac{1}{2} \otimes M^{(\ell_{s+1},\ldots,\ell_N)}$.

Writing
\[ \frac{\sinh(x-t_s-(k+\frac{1}{2})\eta)}{\sinh(x-t_s+(k+\frac{1}{2})\eta)} = \frac{\sinh(x-(t_s + \frac{\eta}{2})-k \eta)}{\sinh(x-(t_s + \frac{\eta}{2})+k \eta)} \frac{\sinh(x-(t_s-k \eta)-\frac{\eta}{2})}{\sinh(x-(t_s-k \eta)+\frac{\eta}{2})} \]
it follows that
\[ j^k_s \overline{\mathcal B}^{\xi,\underline \ell}(x;\mathbf t) = \overline{\mathcal B}^{\xi,\underline \ell'}(x;\mathbf t') j^k_s \]
and hence
\begin{equation} \label{Bbarfusion} j^k_s \overline{\mathcal B}^{\xi,(S),\underline \ell}(\mathbf x;\mathbf t) = \overline{\mathcal B}^{\xi,(S),\underline \ell'}(\mathbf x;\mathbf t') j^k_s. \end{equation}
Since $j^k_s \Omega^{\underline \ell} = \Omega^{\underline \ell'}$ (see Proposition \ref{intertwiners}) it now follows from
\eqref{compatibleF} that
\[
j_s^kf_S^{\underline{\ell}}(\mathbf{t})=f_S^{\underline{\ell}'}(\mathbf{t}')=
j_s^k\textup{Fus}_s^{\underline{\ell}}(f_S^{\underline{\ell}'})(\mathbf{t})
\]
as meromorphic $M^{(\ell_1,\ldots,\ell_{s-1})}\otimes V^k\otimes V^{\frac{1}{2}}\otimes M^{(\ell_s,\ldots,\ell_N)}$ valued
functions in $\mathbf{t}\in\mathbb{C}^N$, which proves the result.
\end{proof}

\begin{rema}
Note that $\sum_{r=1}^N\ell_r=\sum_{r=1}^{N+1}\ell_r'$ for $\underline{\ell}\in\mathbb{C}^N$ with $\ell_s=k$
and with $\underline{\ell}'$ given by \eqref{expanded}. 
Hence the region of meromorphic convergence \eqref{convergencecondition} for the solutions $f_S^{\underline{\ell}}$ and $f_S^{\underline{\ell}'}$ with weight factors \eqref{wghtexplicit} is compatible with fusion.\\
\end{rema}


\end{document}